\documentclass[12pt,a4paper]{amsart}
\usepackage[top=35mm, bottom=35mm, left=30mm, right=30mm]{geometry}
\usepackage{color}
\usepackage[colorlinks=true,citecolor=blue]{hyperref}
\usepackage{txfonts}

\theoremstyle{theorem}
\newtheorem{thm}{Theorem}[section]

\newtheorem{lem}[thm]{Lemma}
\newtheorem{prop}[thm]{Proposition}
\newtheorem{cor}[thm]{Corollary}
\newtheorem{que}[thm]{Question}
\theoremstyle{definition}
\newtheorem{remark}[thm]{Remark}

\begin{document}
\title[A dynamical version of Kuratowski-Mycielski Theorem]
{A dynamical version of Kuratowski-Mycielski Theorem and invariant chaotic sets}
\author[J. Li]{Jian Li}
\address[J. Li]{Department of Mathematics,
Shantou University, Shantou, 515063, Guangdong, China}
\email{lijian09@mail.ustc.edu.cn}

\author[J. L\"u]{Jie L\"u}
\address[J. L\"u]{School of Mathematics, South China Normal University,
Guangzhou 510631, China}
\email{ljie@scnu.edu.cn}
\author[Y. Xiao]{Yuanfen Xiao}
\address[Y. Xiao]{School of Mathematics, South China Normal University,
Guangzhou 510631, China}
\email{xiaoyuanfen@foxmail.com}



\maketitle
\begin{abstract}
We establish a dynamical version of Kuratowski-Mycielski Theorem on
the existence of ``large'' invariant dependent sets.
We apply this result to the study of invariant chaotic sets
in topological dynamical systems, simplify many known results on this topic and
also obtain some new results.
\end{abstract}

\section{Introduction}
By a (topological) dynamical system, we mean a pair  $(X,f)$,
where $X$ is a Polish space and $f\colon X\to X$ is a continuous map.
Let $d$ be a bounded compatible complete metric on $X$.
A subset $A$ of $X$ is called scrambled if for any two distinct points $x,y\in A$,
\[\liminf_{k\to\infty}d(f^k(x),f^k(y))=0\text{ and }
\limsup_{k\to\infty}d(f^k(x),f^k(y))>0.\]
Following the ideas in~\cite{LY75},
a dynamical system $(X,f)$ is called Li-Yorke chaotic if
there exists an uncountable scrambled subset of $X$.
In~\cite{LY75}, Li and Yorke showed that if a continuous map $f\colon [0,1]\to [0,1]$
has a  periodic point of period 3, then it is Li-Yorke chaotic.
Since then the study of chaos theory in topological dynamics
has been attracted lots of attention.

At the beginning, people tried
to construct uncountable scrambled sets directly
and most of the results focused on interval maps.
The turning point came in 2002, when Huang and Ye \cite{HY02} proved that
Devaney chaos implies Li-Yorke chaos and
Blanchard et al. \cite{BGKM02} proved that positive
entropy also implies Li-Yorke chaos.
In \cite{BGKM02}, they also showed the existence of uncountable scrambled sets
by a powerful tool---Mycielski theorem \cite{M64}, which states that
if $R$ is a dense $G_\delta$ subset of $X\times X$ where $X$ is a perfect Polish space, 
then there exists a dense, $\sigma$-Cantor, $R$-dependent subset $A$ of $X$,
where an $R$-dependent subset $A$ means that $(x,y)\in R$
for any two distinct points $x,y\in A$.
In fact, the first application of Mycielski theorem 
to the study of scrambled sets was done in \cite{I91} by Iwanik.
But this method is not widespread until the publication of \cite{BGKM02}.
In \cite{K73} Kuratowski obtained a hyperspace version of Mycielski theorem
by the Baire-category method.
Let us explain the idea of applications of
Kuratowski-Mycielski theorem to the existence of scrambled sets.
For a positive number $\delta>0$, let
\begin{align*}
LY_\delta(X,f)=\Bigl\{(x,y)\in X\times X\colon&
\liminf_{k\to\infty}d(f^k(x),f^k(y))=0\\
&\text{ and }
\limsup_{k\to\infty}d(f^k(x),f^k(y))\geq\delta \Bigr\}
\end{align*}
and
\begin{align*}
sLY(X,f)=\Bigl\{(x,y)\in X\times X\colon&
\liminf_{k\to\infty}d(f^k(x),f^k(y))=0\\
&\text{ and }
\liminf_{k\to\infty}\max\{d(f^k(x),x),d(f^k(y),y)\}=0\Bigr\}.
\end{align*}
It is easy to see that both $LY_\delta(X,f)$ and $sLY(X,f)$
are $G_\delta$ subsets of $X\times X$ and any $LY_\delta(X,f)$-dependent set
or  $sLY(X,f)$-dependent set is scrambled.
Clearly every $\sigma$-Cantor set is uncountable.
In order to show that a dynamical system $(X,f)$ is Li-Yorke chaotic,
by Kuratowski-Mycielski theorem it is sufficient to show that
there exists a perfect subset $Y$ of $X$
such that $sLY(X,f)$ or  $LY_\delta(X,f)$ is dense in $Y\times Y$.
Instead of constructing uncountable scrambled sets directly,
Kuratowski-Mycielski theorem has been applied extensively to show the existence of
``large'' scrambled sets in topological dynamics. 
We refer the reader to \cite{BHS08} and \cite{LY16} for
recent advances on this topic.

In \cite{A04} Akin gave a comprehensive treatment of Kuratowski-Mycielski theorem
and various applications to topological dynamics.
To introduce the result, we need some preparation.
For a  Polish space $X$, let $C(X)$ and $CANTOR(X)$
be the collection of non-empty compact subsets of $X$ and
the collection of Cantor sets of $X$.
A subset $Q$ of $C(X)$ is called \emph{hereditary}
if $A\in Q$ implies that $Q$ contains every compact subset of $A$.
A hereditary subset $Q$ determines a coherent list $\alpha_Q=\{R_n^Q\}_{n\in\mathbb{N}}$ by letting
$R_n^Q\subset X^n$
be the set of $n$-tuples $(x_1,\dotsc,x_n)\in X^n$
such that $\{x_1,\dotsc,x_n\}\in Q$.
We call $\alpha_Q=\{R_n^Q\}_{n\in\mathbb{N}}$
\emph{the coherent list associated with $Q$}.
We say that a subset $A$ of $X$ is \emph{$\alpha_Q$-dependent}
if for any $n\in\mathbb{N}$ and
any pairwise distinct $n$ elements $x_1,x_2,\dotsc,x_n\in A$,
the tuple $(x_1,x_2,\dotsc,x_n)\in R_n^Q$.

\begin{thm}[Kuratowski-Mycielski Theorem]\label{thm:K-M-Theorem}
Let $Q$ be a $G_\delta$, hereditary subset of $C(X)$
for a perfect, Polish space $X$ and
$\alpha_Q=\{R_n^Q\}_{n\in\mathbb{N}}$
 be the coherent list on $X$ associated with $Q$.
The following conditions are equivalent:
\begin{enumerate}
\item $Q$ is a dense subset of $C(X)$;
\item there exists a dense sequence $\{A_i\}$ in $CANTOR(X)$
such that
\[\bigcup_{i=1}^N A_i\in Q,\quad \forall N\geq 1;\]
\item $R_n^Q$ is a dense  $G_\delta $ subset of $X^n$ for $n=1, 2, \dotsc$;
\item there exists a dense, countable, $\alpha_Q$-dependent set in $X$;
\item there exists a dense, $\sigma$-Cantor, $\alpha_Q$-dependent set in $X$;
\item there exists a nowhere meager $\alpha_Q$-dependent set in $X$.
\end{enumerate}
\end{thm}

Note that in Theorem \ref{thm:K-M-Theorem},
  (3)$\Longrightarrow$(5) was proved in \cite{M64} by Mycielski
and (3)$\Longrightarrow$(1) was proved in \cite{K73} by Kuratowski.
The equivalences of (1)$\Longleftrightarrow$(2)$\Longleftrightarrow$(3)$\Longleftrightarrow$(4)$\Longleftrightarrow$(5) in this version
were proved in \cite[Theorem 5.10]{A04}.
(6)$\implies$(3) is obvious and (3)$\implies$(6) was implicitly contained in \cite{BH87}
under the condition of Continuum Hypothesis, see also \cite{TXL07} or \cite{BHS08}.
Without the condition of Continuum Hypothesis,
(3)$\implies$(6) was proved in \cite{BZ15} for a  binary relation
 and in \cite{MRZ17} for a collection of relations recently.
So in \cite[Proposition 53]{BHS08} we can drop the condition of Continuum Hypothesis,
in other words, if a dynamical system $(X,f)$ is generically chaotic, 
then there exists a nowhere meager scrambled set.

In the definition of Li-Yorke chaos, 
 it only requires the existence of an uncountable scrambled set.
It is interesting to know more properties on scrambled sets.
In \cite{D05} Du initiated the study of invariant scrambled sets.
Let $(X,f)$ be a dynamical system.
We say that a subset $A$ of $X$ is \emph{$f$-invariant} if $f(A)\subset A$.
It is shown in \cite{D05} that an interval map has positive topological
entropy if and only if some of its iterates possess
an uncountable invariant scrambled set.
Since then the study of invariant scrambled sets has been paid much attention.
Let us list some.
In \cite{YL09} Yuan and L\"u showed that if a transitive system
has a fixed point then it has a dense, $\sigma$-Cantor, invariant scrambled set.
In \cite{BGO10} Balibrea, Guirao and Oprocha showed that
if a strongly mixing system has a fixed point then it has a dense,
$\sigma$-Cantor, invariant $\delta$-scrambled set for some $\delta>0$.
In \cite{FOW14} Fory\'s, Oprocha and Wilczy\'nski showed that
if a compact dynamical system has the specification property
and has a fixed point then it has a dense, $\sigma$-Cantor, invariant distributionally
$\delta$-scrambled set for some $\delta>0$.
In \cite{FHLO16} Fory\'s et al. showed that  a transitive compact dynamical system
 has a  dense, $\sigma$-Cantor, invariant $\delta$-scrambled set
for some $\delta>0$ if and only if it has a fixed point and is not uniformly rigid.

As we require scrambled sets to be invariant, we can not
use Kuratowski-Mycielski Theorem directly and some addition techniques are required to construct special relations on the space.
Recently, in~\cite{T16} Tan extended Mycielski's Theorem to
$G_\delta$, $f$-invariant relation strings and then applied this result
to invariant scrambled sets.
But the $f$-invariance of a relation string is a quite strong condition,
usually it is not easy to construct such relation string.
The main aim of the paper is to show
the following extension of Kuratowski-Mycielski Theorem.

\begin{thm}[A dynamical version of Kuratowski-Mycielski Theorem]
\label{thm:Main-resut-KWT}
Let $(X,f)$ be a dynamical system with $X$ a perfect Polish space.
Let $Q$ be a $G_\delta$, hereditary subset of $C(X)$
and $\alpha_Q=\{R_n^Q\}_{n\in\mathbb{N}}$ be the coherent list on $X$
associated with $Q$.
The following conditions are equivalent:
\begin{enumerate}
\item $\{A\in C(X)\colon \bigcup_{i=0}^N f^i(A)\in Q,\ N=1,2,\dotsc\}$
is a dense $G_\delta$ subset of $C(X)$;
\item there exists a dense sequence $\{A_i\}$ in $CANTOR(X)$
such that
\[\bigcup_{j=0}^N\bigcup_{i=1}^N f^j(A_i)\in Q ,\quad \forall N\geq 1;\]
\item $\{(x_1,\dotsc,x_n)\in X^n\colon
\bigcup_{i=0}^N f^i(\{x_1,\dotsc,x_n\})\in Q,\ N=1,2,\dotsc\}$
is a dense $G_\delta $ subset of $X^n$ for $n=1, 2, \dotsc$;
\item there exists a dense, countable, $f$-invariant,
$\alpha_Q$-dependent set in $X$;
\item there exists a $c$-dense, $F_\sigma$, $f$-invariant,
$\alpha_Q$-dependent set in $X$;
\item there exists a nowhere meager, $f$-invariant, $\alpha_Q$-dependent set in $X$.
\end{enumerate}
\end{thm}

If $f$ is just the identity map, then Theorem~\ref{thm:Main-resut-KWT}
is nothing but the Kuratowski-Mycielski Theorem.
We will prove Theorem~\ref{thm:Main-resut-KWT} in Section 2
and compare our result with Tan's result.
In Section 3, we apply Theorem~\ref{thm:Main-resut-KWT}
to the study of invariant scrambled sets.
We give a uniform treatment of invariant uniformly chaotic sets,
invariant uniformly mean chaotic sets, invariant $\delta$-scrambled sets
and distributionally $\delta$-scrambled sets,
which extend many results in this topic.

\section{The dynamical version of Kuratowski-Mycielski Theorem}
Let $X$ be a Polish space and $d$ be a bounded compatible complete metric on $X$.
For $n\geq 1$, denote $X^n=X\times X\times \dotsb\times X$ ($n$-times)
and define
$\Delta_n=\{(x, x, \cdots, x)\in X^n\colon x\in X\}$
and $FAT\Delta_n=\{(x_1, x_2, \cdots, x_n)\colon
\text{ there are }i\neq j \text{ such that }x_i=x_j\}$.
It should be noticed that $\Delta_1=X$ and $FAT\Delta_1=\emptyset$.

The \emph{distance} from a point $x$ to a set $A$ is given by
\[d(x, A)=\inf \{d(x, a)\colon a\in A \}.\]
For $\varepsilon>0$, we define the
\emph{$\varepsilon $-neighborhood of a subset $A$} by
\[V_\varepsilon(A)=\{x\in X: d(x, A)<\varepsilon \}.\]
The \emph{Hausdorff metric} of two subsets $A$ and $B$ is given by
\[d_H(A, B)=\inf \{\varepsilon>0 \colon
A\subseteq V_{\varepsilon }(B) \text{ and } B\subseteq V_{\varepsilon}(A)\}.\]

We say that a subset $A$ of $X$ is a \emph{Cantor set} if
it is homeomorphic to the standard middle-third Cantor set,
\emph{$\sigma$-Cantor} if it can be expressed as a countable union of Cantor sets,
\emph{$c$-dense} if $A\cap U$ has the cardinality of continuum
for any non-empty open subset $U$ of $X$,
and \emph{nowhere meager} if $A\cap U$ is not meager
for any non-empty open subset $U$ of $X$.
In \cite{BGKM02}, $\sigma$-Cantor sets are called Mycielski sets,
but in \cite{A04} Mycielski sets are required to be locally non-compact.
In this paper, we adopt the unambiguous terminology---$\sigma$-Cantor.

We consider the following collections of subsets of $X$:
\begin{align*}
  C(X)&=\{A\subset X\colon A \text{ is non-empty compact}\}, \\
  FIN(X)&=\{A\subset X\colon A\text{ is non-empty finite}\}, \\
  CANTOR(X)&=\{A\subset X\colon A\text{ is a Cantor set}\}.
\end{align*}
It is well known that with the Hausdorff metric,
$C(X)$ is a Polish space under the condition that $X$ is Polish.
We will need the following useful lemma.
\begin{lem}[{\cite[Lemma~4.1]{A04}}] \label{lem:maps-on-C(X)}
Let $X$ be a Polish space and $n\in\mathbb{N}$.
 \begin{enumerate}
  \item The map $i_n\colon X^n\to C(X)$,
  $(x_1, x_2, \dotsc, x_n)\mapsto \{x_1, x_2, \dotsc, x_n\}$ is continuous;
  \item The map  $\vee_n\colon C(X)^n\to C(X)$,
  $(A_1, A_2, \dotsc, A_n)\mapsto A_1\cup A_2\cup \dotsb \cup A_n$
 is a continuous and  open surjection;
  \item The map  $\times_n\colon C(X)\to C(X^n)$,  $A\mapsto A^n$ is continuous;
 \item If $f\colon X\to X$ is continuous, then the map $f^*\colon C(X)\to C(X)$,  $A\mapsto f(A) $ is continuous.
\end{enumerate}
\end{lem}

We say that a subset $Q$ of $C(X)$ is  \emph{hereditary} if it satisfies
\[A\in Q \implies C(A)\subseteq Q, \]
\emph{finitely hereditary} if it satisfies
\[A\in Q \implies FIN(A)\subseteq Q, \]
and \emph{finitely determined} if it satisfies
\[A\in Q \iff FIN(A)\subseteq Q. \]
Clearly, we have for $Q\subset C(X)$ that
\begin{center}
finitely determined $\implies$ hereditary $\implies$ finitely hereditary.
\end{center}
Assume that $Q$ is a finitely hereditary subset of $C(X)$.
For $n=1,2,\dotsc$ we define $R_n^Q\subset X^n$
to be the set of $n$-tuples $(x_1,\dotsc,x_n)\in X^n$
such that $\{x_1,\dotsc,x_n\}\in Q$.
We call $\alpha_Q=\{R_n^Q\}_{n\in\mathbb{N}}$
\emph{the coherent list associated with $Q$}.

\begin{remark}\label{rem:G-delta-coherent-list}
Let $Q$ be a hereditary subset of $C(X)$
for a Polish space $X$ and $\alpha_Q=\{R_n^Q\}_{n\in\mathbb{N}}$
 be the coherent list on $X$ associated with $Q$.
It is clear that $R_n^{Q}=i_n^{-1}(Q)$.
If $Q$ is a $G_\delta$ subset of $C(X)$, 
then each $R_n^{Q}$ is a $G_\delta$ subset of $X^n$ 
for $n=1,2,\dotsc$,
because each map $i_n$ is continuous.
\end{remark}

A sequence $\alpha=\{R_n\}_{n\in\mathbb{N}}$ of relations on $X$
is called a \emph{relation string}
if $R_n$ is an $n$-th relation on $X$ for any $n\geq 1$.
We say that a subset $A$ of $X$ is a \emph{dependent set for $\alpha$}
or \emph{$\alpha$-dependent}
if for any $n\in\mathbb{N}$ and
any pairwise distinct $n$ elements $x_1,x_2,\dotsc,x_n\in A$,
the tuple $(x_1,x_2,\dotsc,x_n)\in R_n$.
We use $\mathbf{D}(\alpha)$ to denote the collection of all $\alpha$-dependent sets  and put $\mathbf{C}(\alpha)=\mathbf{D}(\alpha)\cap C(X)$.
In other words, $\mathbf{C}(\alpha)$ is the collection of all compact $\alpha$-dependent sets.
We say that a relation string $\alpha=\{R_n\}$ on $X$ is $G_\delta$
if each $R_n$ is a $G_\delta$ subset of $X^n$ for $n=1,2,\dotsc$.
\begin{lem}\label{lem:G-delta-relation}
If $\alpha=\{R_n\}$ is a $G_\delta$, relation string on a Polish space $X$,
then the collection of compact $\alpha$-dependent sets, $\mathbf{C}(\alpha)$,
is a $G_\delta$, finitely determined subset of $C(X)$.
\end{lem}
\begin{proof}
First note that a subset $A$ of $X$ is $\alpha$-dependent if and only if
$A^n\subset R_n\cup FAT\Delta_n$ for $n=1,2,\dotsc$.
Then $\mathbf{C}(\alpha)$ is finitely determined.
As $R_n$ is a $G_\delta$ subset $X^n$ and $FAT\Delta_n$ is closed
  $R_n\cup FAT\Delta_n$ is $G_\delta$.
By \cite[Proposition 4.3]{A04}, $\{A\in C(X)\colon A^n\subset R_n\cup FAT\Delta_n\}$ is $G_\delta$.
As each $R_n$ is $G_\delta$ for $n=1,2,\dotsc$,
$\mathbf{C}(\alpha)=
\bigcap_{n=1}^\infty\{A\in C(X)\colon A^n\subset R_n\cup FAT\Delta_n\}$
is $G_\delta$.
\end{proof}

\begin{remark}	

\begin{enumerate}
  \item Assume that
  $\alpha=\{R_n\}$ is a $G_\delta$  relation string on $X$.
  By Lemma~\ref{lem:G-delta-relation}, $\mathbf{C}(\alpha)$
  is a $G_\delta$, finitely determined subset of $C(X)$.
  Let $Q=\mathbf{C}(\alpha)$ and
  $\alpha_Q=\{R_n^Q\}$ be the  coherent list associated with $Q$.
Then it is easy to see that $R_n\subset R_n^Q\subset R_n\cup FAT\Delta_n$
for every $n\in\mathbb{N}$
and a subset of $X$ is $\alpha$-dependent
if and only if it is $\alpha_Q$-dependent.
So if we start from a $G_\delta$, relation string $\alpha$ on $X$, 
then we can apply Theorem~\ref{thm:K-M-Theorem}
to the $G_\delta$, finitely determined
subset $\mathbf{C}(\alpha)$ of $C(X)$ to get the properties of
$\alpha$-dependent sets.
\item  Let $\{R_\lambda\}_{\lambda\in\Lambda}$ be a
 collection of countable relations on $X$,
that is, for each $\lambda\in\Lambda$, there exists $k_\lambda\geq 1$ such that $R_\lambda\subset X^{k_\lambda}$. It generates naturally a relation string $\alpha=\{R'_n\}_{n\in\mathbb{N}}$ as follows:
for each $n\in\mathbb{N}$, if there exists a $\lambda$ such that $k_\lambda=n$
, then put $R_n=\bigcap\{R_{k_\lambda}\colon k_\lambda=n\}$,
otherwise put $R_n=X^n$.
Then a subset of $X$ is $R_\lambda$-dependent for all $\lambda\in\Lambda$
if and only if it is $\alpha$-dependent.
Moreover if each $R_\lambda$ is $G_\delta$,
then $\alpha$ is a $G_\delta$ relation string.
So we can also apply Theorem~\ref{thm:K-M-Theorem}
to a collection of countable relations on $X$.
\end{enumerate}
\end{remark}

Now we are ready to prove the dynamical version of Kuratowski-Mycielski Theorem.

\begin{proof}[Proof of Theorem \ref{thm:Main-resut-KWT}]
Let $Q^f=\{A\in C(X)\colon \bigcup_{i=0}^N f^i(A)\in Q,\ N=1,2,\dotsc\}$
and $a_Q^f=\{R_n^{Q,f}\}$ be the  coherent list on $X$ associated with $Q^f$.
It is clear that an $n$-tuples $(x_1,\dotsc,x_n)\in R_n^{Q,f}$
if and only if $\bigcup_{i=0}^N f^i(\{x_1,\dotsc,x_n\})\in Q$ for all $N=1,2,\dotsc$.
It is clear that $Q^f$ is hereditary.
First we have the following Claim.

\medskip
\noindent\textbf{Claim 1}: $Q^f$ is a $G_\delta$ subset of $C(X)$
and $R_n^{Q,f}$ is a $G_\delta $ subset of $X^n$ for $n=1, 2, \dotsc$.
\begin{proof}[Proof of the Claim 1]
For any $N\in\mathbb{N}$, put
$Q_N^f=\{A\in C(X)\colon \bigcup_{i=0}^{N-1} f^i(A)\in Q\}$.
Then $Q^f=\bigcap_{N=1}^\infty Q_N^f$ and
it is enough to show that each $Q_N^f$ is  a $G_\delta$ subset of $C(X)$.
We define several maps as follows:
\begin{align*}
  I_N&\colon C(X)\to C(X)^N,\ A\mapsto (A,A,\dotsc,A),\\
  f^{\times N}&\colon C(X)^N\to C(X)^N,\ (A_1,A_2,\dotsc,A_N)
  \mapsto (A_1,f(A_2),\dotsc,f^{N-1}(A_N)),\\
  \Theta_N&\colon C(X)\to C(X),\ A\mapsto \bigcup_{i=0}^{N-1} f^i(A).
\end{align*}
It is clear that $I_N$ is continuous and by Lemma~\ref{lem:maps-on-C(X)}~(4)
$f^{\times N}$ is continuous.
Note that $\Theta_N= \vee_N\circ f^{\times N}\circ I_N$
and then it is continuous.
Since $Q_N^f=\Theta_N^{-1}(Q)$ 
and $Q$ is a $G_\delta$ subset of $C(X)$,
$Q_N^f$ is  a $G_\delta$ as subset of $C(X)$.
Then $Q^{f}$  is also a $G_\delta$ as subset of $C(X)$.
For each $n\in\mathbb{N}$, 
it is clear that $R_n^{Q,f}=i_n^{-1}(Q^f)$.
Then $R_n^{Q,f}$ is a $G_\delta$ subset of $X^n$.
\end{proof}

\medskip
\noindent\textbf{Claim 2}:
A subset $A$ of $X$ is $\alpha_Q^f$-dependent
if and only if $\bigcup_{i=0}^\infty f^i(A)$ is $\alpha_Q$-dependent.
In particular, every subset of an $f$-invariant $\alpha_Q$-dependent set
is $\alpha_Q^f$-dependent.
\begin{proof}[Proof of the Claim 2]
Assume that $A$ is $\alpha_Q^f$-dependent and
let $\widehat{A}=\bigcup_{j=0}^\infty f^j(A)$.
For every $n\in\mathbb{N}$ and any pairwise distinct $n$
elements $\widehat{x}_1,\widehat{x}_2,\dotsc,\widehat{x}_n\in \widehat{A}$,
there exist $n$ elements
$x_1,x_2,\dotsc,x_n\in A$ and $k_1,k_2,\dotsc,k_n$ such that
$f^{k_i}(x_i)=\widehat{x}_i$.
Let $m$ be the cardinal number of the set $\{x_1,x_2,\dotsc, x_n\}$.
Without loss of generality, assume that $x_1,x_2,\dotsc,x_m$ are pairwise distinct.
As $A$ is $\alpha_Q^f$-dependent, the tuple $(x_1,x_2,\dotsc,x_m)\in R_m^{Q,f}$.
By the definition of $R_m^{Q,f}$, we have
$\bigcup_{i=0}^N f^i (\{x_1,x_2,\dotsc,x_m\})\in Q$ for all $N\geq 1$.
As $Q$ is hereditary, we get
$\{\widehat{x}_1,\widehat{x}_2,\dotsc,\widehat{x}_n\}\in Q$
and then $(\widehat{x}_1,\widehat{x}_2,\dotsc,\widehat{x}_n)\in R_n^Q$.
This implies that $\widehat{A}$ is $\alpha_Q$-dependent.

Now assume that $\bigcup_{i=0}^\infty f^i(A)$ is $\alpha_Q$-dependent.
As $Q$ is hereditary,
every finite subset of $\bigcup_{i=0}^\infty f^i(A)$ is contained by  $Q$.
Fix $n\in\mathbb{N}$ and any pairwise distinct $n$
elements $x_1,x_2,\dotsc,x_n\in A$.
For any $N\geq 1$, $\bigcup_{i=0}^N f^i(\{x_1,\dotsc,x_n\})$
is a finite subset of $\bigcup_{i=0}^\infty f^i(A)$
and then is contained by $Q$.
Therefore, $(x_1,x_2,\dotsc,x_n)\in R_n^{Q,f}$ and $A$ is $\alpha_Q^f$-dependent.
\end{proof}
Applying Theorem~\ref{thm:K-M-Theorem} to the $G_\delta$,
hereditary subset $Q^f$ of $C(X)$, we get the following equivalent conditions:
\begin{enumerate}
\item[($1'$)] $Q^f$ is a dense subset of $C(X)$;
\item[($2'$)] there exists a dense sequence $\{A_i\}$ in $CANTOR(X)$
such that
\[\bigcup_{i=1}^N A_i\in Q^f,\quad \forall N\geq 1;\]
\item[($3'$)] $R_n^{Q,f}$ is a dense  $G_\delta $
subset of $X^n$ for $n=1, 2, \dotsc$;
\item[($4'$)] there exists a dense, countable, $\alpha_Q^f$-dependent set in $X$;
\item[($5'$)] there exists a dense, $\sigma$-Cantor, $\alpha_Q^f$-dependent set in $X$;
\item[($6'$)] there exists a
nowhere meager $\alpha_Q^f$-dependent set in $X$.
\end{enumerate}

We are going to show that ($i$)$\Leftrightarrow$($i'$) for all $i=1,2,\dotsc, 6$.
By the definitions of $Q^f$ and $\alpha_Q^f$,
it is clear that ($i$)$\Leftrightarrow$($i'$) for $i=1,2,3$.

Assume that a subset $A$ of $X$ is $\alpha_Q^f$-dependent.
Let $\widehat A=\bigcup_{i=0}^\infty f^i(A)$.
By Claim 2, $\widehat A$ is an $f$-invariant $\alpha_Q$-dependent set.
As $A\subset \widehat A$,
 if $A$ is countable dense or nowhere meager then so is $\widehat A$.
So we get ($i'$)$\Rightarrow$($i$) for $i=4,6$.
If $A$ is dense and $\sigma$-Cantor, then $\widehat A$ is $F_\sigma$.
Moreover, for every non-empty open subset $U$ of $X$,
$A\cap U$ contains a Cantor set, then $A$ is $c$-dense and so is $\widehat A$.
So we get ($5'$)$\Rightarrow$($5$).

Now assume that a subset $B$ of $X$ is an $f$-invariant $\alpha_Q$-dependent set.
By Claim 2, $B$ is also $\alpha_Q^f$-dependent. So it is clear that
($i$)$\Rightarrow$($i'$) for $i=4,6$.
If $B$ is $c$-dense and $F_\sigma$,
as every uncountable Borel set in a Polish space contains an Cantor set
(see e.g. \cite[Theorem 13.6]{K95}),
there exists a dense, $\sigma$-Cantor subset $C$ of $B$.
By Claim 2 again, $C$ is $\alpha_Q^f$-dependent. So we get ($5$)$\Rightarrow$($5'$).
\end{proof}

\begin{remark}
If in addition we assume that $f|_C$ is injective for all $C\in Q$,
then for every Cantor set $A$, $f(A)$ is also a Cantor set.
Therefore, in condition (2) of Theorem~\ref{thm:Main-resut-KWT}, the set
$\bigcup_{j=0}^\infty \bigcup_{i=1}^\infty f^j(A_i)$
is a dense, $\sigma$-Cantor, $f$-invariant, $\alpha_Q$-dependent set
and then the condition of (5)
in Theorem~\ref{thm:Main-resut-KWT} can be replaced by
\begin{enumerate}
  \item[(5$''$)] there exists a dense, $\sigma$-Cantor,
  $f$-invariant, $\alpha_Q$-dependent set.
\end{enumerate}
\end{remark}

Let $A$ be an uncountable subset of a Polish space.
By the Cantor-Bendixson Theorem (see e.g. \cite[Theorem 6.4]{K95})
there exists an uniquely decomposition $\overline{A}=P\cup U$
with $P$ is a perfect subset of $X$ and $U$ is at most countable open subset of $\overline{A}$.
In particular, $A\cap P$ is dense in $P$.
For a $G_\delta$, hereditary subset of $C(X)$,
if there exists an uncountable, $\alpha_Q$-dependent set $A$ in $X$,
from the above observation, we can apply the Kuratowski-Mycielski Theorem to
$C(P)\cap Q$ and then get a Cantor set in $C(P)\cap Q$, see e.g. \cite[Corollary 5.11]{A04}.
Recently, the authors in~\cite{DK16} slightly extended this result by showing that
if for every countable ordinal $\gamma$ there exists an $\alpha_Q$-dependent set $B$ of $X$
 such that the Cantor-Bendixson rank of $B$ is not less than $\gamma$, 
then there exists a Cantor $\alpha_Q$-dependent set.
We refer the reader to \cite[Definition 6.12]{K95} for the definition of Cantor-Bendixson rank.
To sum up, we have the following local version of Kuratowski-Mycielski Theorem.
\begin{cor}\label{cor:local-K-W-Thm}
Let $Q$ be a $G_\delta$, hereditary subset of $C(X)$ for a Polish space $X$
and $\alpha_Q=\{R_n^Q\}$
be the coherent list on $X$ associated with $Q$.
The following conditions are equivalent:
\begin{enumerate}
\item  there exists a Cantor set in $Q$;
\item  there exists an uncountable, $\alpha_Q$-dependent set in $X$;
\item  for every countable ordinal $\gamma$,
 there exists an $\alpha_Q$-dependent set $B$ of $X$
 such that the Cantor-Bendixson rank of $B$ is not less than $\gamma$.
\end{enumerate}
\end{cor}

Similar to the proof of Theorem~\ref{thm:Main-resut-KWT},
applying Corollary~\ref{cor:local-K-W-Thm} to the $G_\delta$, hereditary
subset $Q^f$ of $C(X)$, we get the following result.

\begin{cor}\label{cor:local-f-K-W-Thm}
Let $(X,f)$ be a dynamical system.
Let $Q$ be a $G_\delta$ hereditary subset of $C(X)$.
and $\alpha_Q=\{R_n^Q\}$
be the coherent list on $X$ associated with $Q$.
The following conditions are equivalent:
\begin{enumerate}
\item  there exists a Cantor set $A$ such that
$\bigcup_{i=0}^N f^i(A)\in Q$ for $N=1,2,\dotsc$;
\item  there exists an uncountable, $f$-invariant, $\alpha_Q$-dependent set in $X$;
\item  for every countable ordinal $\gamma$,
 there exists an $f$-invariant $\alpha_Q$-dependent set $B$ of $X$
 such that the Cantor-Bendixson rank of $B$ is not less than $\gamma$.
\end{enumerate}
\end{cor}

\begin{remark}
If in addition we assume that $f|_C$ is injective for all $C\in Q$,
then under the condition (1) of Corollary~\ref{cor:local-f-K-W-Thm},
the set $\bigcup_{j=0}^\infty f^j(A)$ is a $\sigma$-Cantor,
$f$-invariant, $\alpha_Q$-dependent set.
Let $Y=\overline{\bigcup_{j=0}^\infty f^j(A)}$.
Then $Y$ is a perfect, $f$-invariant subset of $X$ and
$Q\cap C(Y)$ is a dense $G_\delta$ subset of $Y$.
\end{remark}

In \cite{T16}, Tan extended  Mycielski's Theorem to
$G_\delta$, $f$-invariant relation strings.
To introduce the result, we need the following concepts.
For any $n\geq 1$,
an \emph{$n$-tuple permutation} is
a bijection from the set $\{1,2,\dotsc,n\}$ onto itself.
For an $n$-tuple permutation $\gamma$,
we define a map $T^n_\gamma\colon X^n\to X^n$ by
\[T_\gamma^n(x_1,x_2,\dotsc,x_n)=(x_{\gamma(1)},x_{\gamma(2)},\dotsc,x_{\gamma(n)})\]
for any $(x_1,x_2,\dotsc,x_n)\in X^n$.
For any $n\geq 1$ and $1\leq j\leq n+1$,
we define a map $\xi^{(n+1)}_j\colon X^{n+1}\to X^n$ by
\[\xi^{(n+1)}_j(x_1,x_2,\dotsc,x_{n+1})=(x_1,\dotsc,x_{j-1},x_{j+1},\dotsc,x_{n+1})\]
for any $(x_1,x_2,\dotsc,x_{n+1})\in X^{n+1}$,
in other words, $\xi^{(n+1)}_j$ is a projection from $X^{n+1}$ to $X^n$ erasing
the $j$-th coordinate.

Let $\alpha=\{R_n\}$ be a relation string on $X$ and
$f\colon X\to X$ be a continuous map.
We say that $\alpha$ is \emph{$f$-invariant} if for each $n\geq 1$
and any $(x_1,\dotsc, x_n)\in R_n$, we have
\[T^{n+1}_{\gamma}(x_1,\dotsc,x_n,f^i(x_1))\in R_{n+1}\]
and
\[T^{n}_{\gamma'}(\xi^{(n+1)}_j(x_1,\dotsc,x_n,f^i(x_1)))\in R_n\]
where $\gamma$ (resp.\ $\gamma'$)
runs over all $(n+1)$-tuple (resp.\ $n$-tuple) permutations,
 $i\geq 1$ and $1\leq j\leq n+1$.
The extended Mycielski's Theorem in \cite{T16} is as follows.
\begin{thm}[{\cite[Theorem 5.5]{T16}}] \label{thm:Tan-result}
Let $X$ be a perfect Polish space and $f\colon X\to X$ be a continuous map.
If $\alpha=\{R_n\}_{n\in\mathbb{N}}$ is a $G_\delta$,
$f$-invariant relation string on $X$,
then the following conditions are equivalent:
\begin{enumerate}
\item $R_n$ is dense in $X^n$ for $n=1,2,\dotsc$;
\item there exists a dense subset $A$ of $X$
such that $\bigcup_{j=0}^\infty f^j(A)$ is $\alpha$-dependent;
\item there exists a dense sequence $\{A_i\}$ in $CANTOR(X)$
such that $\bigcup_{i=1}^\infty\bigcup_{j=0}^\infty f^j(A_i)$ is $\alpha$-dependent;
\item there exists a dense, $\sigma$-Cantor subset $M$ in $X$
such that $\bigcup_{j=0}^\infty f^j(M)$ is $\alpha$-dependent.
\end{enumerate}
\end{thm}

\begin{remark}
Assume that $\alpha=\{R_n\}_{n\in\mathbb{N}}$ is a
$f$-invariant relation string on $X$.
If $A$ is $\alpha$-dependent, by~\cite[Lemma 3.4]{T16}
$A\cup f(A)$ is also $\alpha$-dependent.
By~\cite[Lemma 3.4]{T16} again,
$A\cup f(A)\cup f^2(A)=(A\cup f(A))\cup f(A\cup f(A))$ is also $\alpha$-dependent.
By induction we can conclude that  $\bigcup_{j=0}^N f^j(A)$ is $\alpha$-dependent for any $ N=1,2,\dotsc $. Thus, by the definition of $\alpha$-dependent set, we have 
$\bigcup_{j=0}^\infty f^j(A)$ is $\alpha$-dependent.
Therefore, Theorem~\ref{thm:Tan-result} follows from this observation
and Theorem~\ref{thm:K-M-Theorem}.
\end{remark}

In the next section we will see that, in some situations, we can apply Theorem ~\ref{thm:Main-resut-KWT} to determine the existence of invariant scrambled sets, meanwhile, it is not easy to check the condition ``$f$-invariant relation string" in Theorem~\ref{thm:Tan-result}. 

\section{Invariant scrambled sets in topological dynamical systems}
In this section, we apply the dynamical version of Kuratowski-Mycielski Theorem
to the study of invariant scrambled sets 
in topological dynamical systems.
First, let us recall some preliminaries in topological dynamics.

\subsection{Topological dynamics}
By a (topological) dynamical system, we mean a pair  $(X,f)$,
where $X$ is a Polish space and $f\colon X\to X$ is a continuous map.
We say that a dynamical system $(X,f)$ is \emph{non-trivial}
if $X$ is not a singleton.
A subset $A$ of $X$ is \emph{$f$-invariant} if $f(A)\subset A$.
If a closed subset $Y$ of $X$ is $f$-invariant,
then the restriction $(Y,f|_Y)$ of $(X,f)$ to $Y$ is itself a dynamical system,
and we will call it a subsystem of $(X,f)$.
If there is no ambiguity, we will denote the restriction $f|_Y$ by $f$ for simplicity.

We say that a point $x\in X$ is a \emph{fixed point} if $f(x)=x$,
a \emph{periodic point} if $f^n(x)=x$ for some $n\in\mathbb{N}$,
and a \emph{recurrent point} if $\liminf_{n\to\infty} d(f^n(x),x)=0$.
We denote the \emph{orbit} of $x$ by
\[Orb(x,f)=\{x,f(x),f^2(x),\dotsc\},\]
and the $\omega$-limit set of $x$ by
\[\omega(x,f)=\bigcap_{n=1}^\infty \overline{Orb(f^n(x),f)}.\]

We say that a dynamical system $(X,f)$ is \emph{minimal} if
there is no proper subsystems.
It is clear that $(X,f)$ is minimal if and only if $\omega(x,f)=X$ for all $x\in X$.
A point $x\in X$ is called a \emph{minimal point} if it is contained in a minimal
subsystem of $(X,f)$, in other words, $(\overline{Orb(x,f)},f)$ is a minimal system.

A dynamical system $(X,f)$ is called \emph{transitive}
if for every two non-empty open subsets $U$ and $V$ of $X$
there is an $n\in\mathbb{N}$ such that $U\cap f^{-n}(V)\neq\emptyset$.
A point $x\in X$ is called \emph{a transitive point} if $\omega(x,f)=X$.
In our setting, as $X$ is a Polish  space,
$(X,f)$ is transitive if and only if the collection
of transitive points is a dense $G_\delta$ subset of $X$.
For a transitive system $(X,f)$,
if $X$ has an isolated point, then the system consists of just
one periodic orbit. So, for a non-periodic transitive system
$(X, f)$, the space $X$ is always perfect.

For any $n\in\mathbb{N}$,
the $n$-fold product system of $(X,f)$ is denoted by $(X^n,f^{(n)})$,
where $f^{(n)}=f\times f\times \dotsb\times f$ ($n$-times).
A dynamical system $(X,f)$ is called \emph{weakly mixing}
if the product system $(X^2,f^{(2)})$ is transitive,
and \emph{strongly mixing} if for every two non-empty open subsets $U$ and $V$ of $X$
there is an $N\in\mathbb{N}$ such that $U\cap f^{-n}(V)\neq\emptyset$
for all $n\geq N$.
It is clear that strong mixing implies weak mixing,
which in turn implies transitivity.

We say that a dynamical system $(X,f)$ is  \emph{compact}
if the state space $X$ is compact.
By the Zorn's Lemma, it is not hard to see that every compact dynamical system
has a minimal subsystem and then there exists some minimal point.

\subsection{Invariant uniformly chaotic sets}
Let $(X,f)$ be a dynamical system.
We say that a subset $A$ of $X$
is \emph{uniformly proximal} if
$\liminf_{k\to\infty}diam(f^k(A))=0$ where $ diam(\cdot) $ denotes the diameter of a set.
We let $Q(PROX,f)\subset C(X)$ denote the set of compact
uniformly proximal subsets of $X$
and $\{PROX_n(f)\}$ the coherent list on $X$ associated with $Q(PROX,f)$.
Clearly, $Q(PROX,f)$ is a hereditary subset of $C(X)$ and
a tuple $(x_1,\dotsc,x_n)\in PROX_n(f)$ if and only if it is proximal,
that is $\liminf_{k\to \infty} \max_{1\leqslant i<j\leqslant n} d(f^k(x_i), f^k(x_j))=0$.
$\{PROX_n(f)\}$-dependent sets are called \emph{proximal sets}.
So a subset $A$ of $X$
is proximal if and only if every finite subset of $A$ is uniformly proximal.

We say that a subset $A$ of $X$ is
\emph{uniformly recurrent} if  for every $\varepsilon>0$
there exists $k\geq 1$ such that $d(f^k(x),x)<\varepsilon$ for all $x\in A$.
We let $Q(RECUR,f)\subset C(X)$ denote the set
of compact uniformly recurrent subsets of $X$
and $\{RECUR_n(f)\}$ the coherent list on $X$ associated with $Q(RECUR,f)$.
Clearly, $Q(RECUR,f)$ is a hereditary subset of $C(X)$ and
a tuple $(x_1,\dotsc,x_n)\in RECUR_n(f)$ if and only if it is recurrent
in $(X^n,f^{(n)})$,
that is, for every $\varepsilon>0$ there exists $k\geq 0$
such that $d(f^k(x_i),x_i)<\varepsilon$ for all $i=1,2,\dotsc,n$.
$\{RECUR_n(f)\}$-dependent sets are called \emph{recurrent sets}.
So a subset $A$ of $X$ is recurrent if and only if
every finite subset of $A$ is uniformly recurrent.

A subset $K\subseteq X $ is called a \emph{uniformly chaotic set}
if there are Cantor sets $C_1\subseteq C_2\subseteq \cdots $ such that
\begin{enumerate}
  \item $K=\bigcup_{i=1}^\infty C_i $ is a recurrent subset of $X$ and
  also a proximal subset of $X$;
  \item for each $i=1, 2, \cdots $,  $C_i$ is uniformly recurrent;
  \item for each $i=1, 2, \cdots $,  $C_i$ is uniformly proximal.
\end{enumerate}

It is shown in \cite{AGHSY10} that for a non-trivial transitive system $(X,f)$,
if there exists some subsystem $(Y,f)$ such that $(X\times Y,f\times f)$ is transitive, 
then there exists a dense uniformly chaotic set.
In particular, if there exists a fixed point,  then
there exists a dense uniformly chaotic set.
Here we show that in fact there exists a dense, $f$-invariant, uniformly chaotic set
in this  case.

\begin{thm}\label{thm:uniform-chaos}
Let $(X,f)$ be a non-trivial transitive system.
If it has a fixed point, then there exists
a dense, $f$-invariant, uniformly chaotic set in $X$.
\end{thm}
\begin{proof}
Let $Q=Q(RECUR,f)\cap Q(PROX,f)$ and
$\alpha_Q=\{R_n^Q\}_{n\in\mathbb{N}}$ be the coherent list on $X$
associated with $Q$.
Note that
\begin{align*}
Q(RECUR,f)&=\bigcap_{M=1}^\infty \bigcap_{N=1}^\infty \bigcup_{k=N}^\infty
\biggl\{A\in C(X)\colon d(f^k(x),x)<\frac{1}{M},\forall x\in A\biggr\},\\
Q(PROX,f)&=\bigcap_{M=1}^\infty \bigcap_{N=1}^\infty \bigcup_{k=N}^\infty
\biggl\{A\in C(X)\colon diam(f^k(A))<\frac{1}{M}\biggr\}.
\end{align*}
It is clear that
both $Q(RECUR,f)$ and $Q(PROX,f)$ are hereditary $G_\delta$ subsets of $C(X)$
(see also pages 38 and 41 in \cite{A04}).
Then $Q$ is also a hereditary $G_\delta$ subset of $C(X)$.
Pick a transitive point $x\in X$ and put $A=Orb(x,f)$.
It is clear that $A$ is  dense and $f$-invariant.
Let $p\in X$ be a fixed point.
As $x$ is a transitive point, there exist two increasing sequences
$\{n_i\}$ and $\{m_i\}$ of positive integers such that
$\lim_{i\to\infty} f^{n_i}(x)=p$ and $\lim_{i\to\infty}f^{m_i}(x)=x$.
By the continuity of $f$, for any $k\geq 1$
$\lim_{i\to\infty} f^{n_i}(f^k(x))=f^k(p)=p$ and
$\lim_{i\to\infty}f^{m_i}(f^k(x))=f^k(x)$.
So every finite subset of $A$ is uniformly recurrent and uniformly proximal,
which implies that $A$ is $\alpha_Q$-dependent.
Then the condition (4) of Theorem~\ref{thm:Main-resut-KWT} is satisfied.
By the condition (2) of  Theorem~\ref{thm:Main-resut-KWT},
 there exists a dense sequence $\{A_i\}$ in $CANTOR(X)$
such that $\bigcup_{j=0}^N\bigcup_{i=1}^N f^j(A_i)\in Q$ for all $N\geq 1$.
As $A_i$ is uniformly recurrent, $f^j|_{A_i}$ is injective,
Then $f^j(A_i)$ is a Cantor set, as so is $A_i$.
Let $C_k=\bigcup_{j=0}^k\bigcup_{i=1}^k f^j(A_i)$ and $C=\bigcup_{k=1}^\infty C_k$.
Then $C$ is a dense, $f$-invariant, uniformly chaotic set.
\end{proof}

\begin{remark}\label{rem:proximal-fixed-point}
If $(X,f)$ is a compact dynamical system and there exists $x\in X$
such that $(x,f(x))$ is proximal,
then there exists an  increasing sequence
$\{n_i\}$ of positive integers such that
$\lim_{i\to\infty} d(f^{n_i}(x),f^{n_i}(f(x)))=0$.
As $X$ is compact,
without loss of generality,  assume that $\lim_{i\to\infty} f^{n_i}(x)=p$.
By the continuity of $f$, $\lim_{i\to\infty} f^{n_i}(f(x))=f(p)$.
This implies that $d(p,f(p))=0$, in other words, $p$ is a fixed point.
Therefore, the existence of a fixed point is a necessary condition
for the existence of invariant proximal sets.
\end{remark}

For compact dynamical systems, we have the following equivalent condition
for the existence of $f$-invariant, uniformly chaotic sets.
\begin{thm}\label{thm:cpt-inv-uniform-chaos}
For a  compact dynamical system $(X,f)$,
the following conditions are equivalent:
\begin{enumerate}
  \item there exists an $f$-invariant, uniformly chaotic set in $X$;
  \item there exists a subsystem $(Y,f)$ which has a dense,
  $f$-invariant, uniformly chaotic set;
  \item there exists a recurrent point $x\in X$ such
that $x\neq f(x)$ and $(x,f(x))$ is proximal;
\item there exists a recurrent point $x\in X$ and a fixed point $p\in X$
  such that $x\neq p$ and $(x,p)$ is proximal.
\end{enumerate}
\end{thm}
\begin{proof}
(1)$\Rightarrow$(2) Assume that $C$ is an $f$-invariant, uniformly chaotic set in $X$.
Let $Y=\overline{C}$. Then $Y$ is closed and $f$-invariant and
$C$ is a dense, $f$-invariant, uniformly chaotic set in the subsystem $(Y,f)$.

(2)$\Rightarrow$(3)
Let $C$ be a dense, $f$-invariant, uniformly chaotic set in $Y$.
As $C$ contains at most one periodic point, pick a non-periodic point $x\in C$.
Then $x\neq f(x)$.
As $C$ is $f$-invariant, $f(x)$ is also in $C$ and then $(x,f(x))$ is proximal.

(3)$\Rightarrow$(4) It follows from Remark~\ref{rem:proximal-fixed-point}.

(4)$\Rightarrow$(1) Let $Y=\overline{Orb(x,f)}$.
As $x$ is recurrent, $(Y,f)$ is transitive.
Note that $p\in Y$ but $x\neq p$. So $x$ is not a periodic point and
then $Y$ is perfect.
Now applying Theorem~\ref{thm:uniform-chaos} to the subsystem $(Y,f)$,
we get a dense, $f$-invariant, uniformly chaotic set in $Y$, which is as required.
\end{proof}

Let $(X, f) $ be a dynamical system and $n\geqslant 2$.
We say that an tuple $(x_1, x_2, \dotsc, x_n)\in X^n\setminus FAT\Delta_n$
is \emph{strongly $n$-scrambled} if it is proximal and a recurrent point in $(X^n,f^{(n)})$.
A subset $S\subseteq X$ (with at least $n$ points)
 is called \emph{strongly $n$-scrambled}
if every $n$ pairwise distinct points form an $n$-$\delta$-scrambled tuple.
Note that strongly $2$-scrambled tuples are
just the classical strongly scrambled pairs.

\begin{remark}
It is shown in \cite[Theorem 2.7]{YL09} that
for a non-trivial transitive system $(X,f)$ if
it has a fixed point, then there exists
a dense, $\sigma$-Cantor,
$f$-invariant, strongly $n$-scrambled set in $X$ for all $n\geq 2$.
Clearly, every uniformly chaotic set is strongly $n$-scrambled for all $n\geq 2$.
So Theorem~\ref{thm:uniform-chaos} slightly extends this result.

In fact, in \cite[Theorem 2.7]{YL09} the condition of existence of a fixed point can
be replaced by the following condition ($*$):
``for every $\varepsilon>0$ there exists a compact $f$-invariant subset $Y$ of $X$
with $diam(Y)<\varepsilon$''.
If $X$ is compact, then it is easy to see that this condition ($*$)
is equivalent to the existence of a fixed point.
But there exists a transitive system which satisfies this condition without
periodic point (see \cite[Theorem 2.8]{YL09}).
Using similar arguments, Theorem~\ref{thm:uniform-chaos} also holds if
replacing the condition of existence of a fixed point by the condition ($*$) above.
\end{remark}

\begin{remark}
It is shown in \cite[Theorem 4.3]{FHLO16} that
for a compact dynamical system $(X,f)$ if
it has an uncountable invariant strongly scrambled set
then it also has a $\sigma$-Cantor, $f$-invariant, strongly scrambled set.
In fact, by Theorem~\ref{thm:cpt-inv-uniform-chaos}
if there exists a point $x\in X$ such that $(x,f(x))$ is a strongly scrambled pair, 
then there exists an $f$-invariant, uniformly chaotic set in $X$.
\end{remark}

\subsection{Invariant uniformly mean chaotic sets}
Let $(X,f)$ be a dynamical system.
We say that a subset $A$ of $X$ is \emph{uniformly mean proximal}
if
\[\liminf_{n\to\infty}\frac{1}{n}\sum_{k=1}^n diam(f^k(A))=0.\]
We let $Q(\overline{PROX},f)\subset C(X)$ denote the set of compact
uniformly mean proximal subsets of $X$
and $\{\overline{PROX}_n(f)\}$ the associated coherent list on $X$.
Clearly, $Q(\overline{PROX},f)$ is a hereditary subset of $C(X)$ and
a tuple $(x_1,\dotsc,x_n)\in \overline{PROX}_n(f)$ if and only if it is mean proximal,
that is \[\liminf_{n\to\infty}\frac{1}{n}\sum_{k=1}^n
\max_{1\leq i<j\leq n} d(f^k(x_i),f^k(x_j))=0.\]
$\{\overline{PROX}_n(f)\}$-dependent sets are called \emph{mean proximal sets}.
So a subset $A$ of $X$ is mean proximal if and only if
every finite subset of $A$ is uniformly mean proximal.

Let $P\subset \mathbb{N}$.
The \emph{upper density of $P$}
is defined by
\[\overline{d}(P)=\limsup_{m\rightarrow\infty}\frac{\#(P\cap\{1,\cdots,m\})}{m},\]
where as usual $\#(A)$ denotes the cardinality of a set $A$.

The following lemma is a folklore result, we refer the reader to
\cite[Theorem 1.20]{W82} or \cite[Lemma 3.1]{LTY15}
for similar results in the same idea.
\begin{lem}\label{lem:mean-proximal}
If $\{a_k\}_{k=1}^\infty$ is a bounded sequence of non-negative numbers,
then the following conditions are equivalent:
\begin{enumerate}
  \item $\liminf_{n\to\infty}\frac{1}{n}\sum_{k=1}^n a_k=0$;
  \item for every $\varepsilon>0$,
the set $\{k\in\mathbb{N}\colon 0\leq a_k<\varepsilon\}$ has  upper density one;
\item there exists a subsequence $\{k_i\}$ of  upper density one
such that $\lim_{i\to\infty} a_{k_i}=0$.
\end{enumerate}
\end{lem}
By Lemma~\ref{lem:mean-proximal},
a subset $A$ of $X$ is mean proximal if and only if
\[\lim_{\varepsilon\to 0^+}\limsup_{n\to\infty}\frac{1}{n}
\#\bigl\{i\colon diam(f^i(A))<\varepsilon, 0\leq i\leq n-1\bigr\}=1.\]
Note that if a subset $A$ satisfies the above formula
then it is called \emph{distributionally proximal} in \cite{O11}.

Combining the ideas of uniformly chaotic set and mean proximal set,
we introduce the concept of uniformly mean chaotic set.
A subset $K\subseteq X $ is called a \emph{uniformly mean chaotic set}
if there are Cantor sets $C_1\subseteq C_2\subseteq \cdots $ such that
\begin{enumerate}
  \item $K=\bigcup_{i=1}^\infty C_i $ is a  recurrent subset of $X$ and
  also a mean proximal subset of $X$;
  \item for each $i=1, 2, \cdots $,  $C_i$ is uniformly recurrent;
  \item for each $i=1, 2, \cdots $,  $C_i$ is uniformly mean proximal.
\end{enumerate}

We have the following criterion for uniformly mean chaos.
\begin{thm}\label{thm:mean-proximal-chaos}
Let $(X,f)$ be a non-trivial transitive system.
If there exists a transitive point $x$ and a fixed point $p$
such that $(x,p)$ is mean proximal,
then there exists a dense, $f$-invariant,
uniformly mean chaotic set in $X$.
\end{thm}
\begin{proof}
Let $Q=Q(RECUR,f)\cap Q(\overline{PROX},f)$ and
$\alpha_Q=\{R_n^Q\}_{n\in\mathbb{N}}$ be the coherent list on $X$
associated with $Q$.
Note that
\[ Q(\overline{PROX},f)= \bigcap_{M=1}^\infty \bigcup_{n=1}^\infty
\biggl\{A\in C(X)\colon \frac{1}{n}\sum_{k=1}^n diam(f^k(A))<\frac{1}{M}\biggr\} \]
It is clear that
$Q(\overline{PROX},f)$ is a hereditary $G_\delta$ subset of $C(X)$ (see also \cite[Theorem 58]{O11}).
Then $Q$ is also a hereditary $G_\delta$ subset of $C(X)$.
Note that the transitive point $x$ is mean proximal to the fixed point $p$.
Put $A=Orb(x,f)$.
It is clear that $A$ is  dense and $f$-invariant.
By the proof of Theorem~\ref{thm:uniform-chaos},
every finite subset of $A$ is uniformly recurrent.
As $(x,p)$ is mean proximal,
there exists a subsequence $\{k_i\}$ of upper density one
such that $\lim_{i\to\infty} d(f^{k_i}(x),p)=0$.
By the continuity of $f$ and $p$ is a fixed point,
for every $n\geq 1$, $\lim_{i\to\infty} d(f^{k_i}(f^{n}(x)),p)=0$.
So every finite subset of $A$ is uniformly mean proximal.
Then $A$ is $\alpha_Q$-dependent
and  the condition (4) of Theorem~\ref{thm:Main-resut-KWT} is satisfied.
By the condition (2) of  Theorem~\ref{thm:Main-resut-KWT},
there exists a dense sequence $\{A_i\}$ in $CANTOR(X)$
such that $\bigcup_{j=0}^N\bigcup_{i=1}^N f^j(A_i)\in Q$ for all $N\geq 1$.
As $A_i$ is uniformly  recurrent, $f^j|_{A_i}$ is injective,
Then $f^j(A_i)$ is a Cantor set, as so is $A_i$.
Let $C_k=\bigcup_{j=0}^k\bigcup_{i=1}^k f^j(A_i)$ and $C=\bigcup_{k=1}^\infty C_k$.
Then $C$ is a dense, $f$-invariant, mean proximal,
uniformly chaotic set.
\end{proof}

For a point $x\in X$, the \emph{mean proximal cell} of $X$ is defined as
\begin{align*}
  \overline{PROX}(x)&=\{y\in X\colon (x,y)\text{ is mean proximal}\}\\
  &=\bigcap_{M=1}^\infty \bigcup_{n=1}^\infty
\biggl\{y\in X\colon \frac{1}{n}\sum_{k=1}^n  d(f^k(x),f^k(y))<\frac{1}{M}\biggr\}
\end{align*}
It is clear that that $\overline{PROX}(x)$ is a $G_\delta$ subset of $X$.

We give some examples which are uniformly mean chaotic.
Recall that a dynamical system $(X,f)$ is \emph{exact} if
for every non-empty open subset $U$ of $X$,
there exists $n\geq 0$ such that $f^n(U)=X$.
It is clear that every exact system is strongly mixing.

\begin{prop}
If non-trivial dynamical system $(X,f)$ is exact and has a fixed point $p$,
then there exists a dense, $f$-invariant,
uniformly mean chaotic set in $X$.
\end{prop}
\begin{proof}
As $(X,f)$ is exact, for every non-empty open subset $U$ of $X$ there exists $z\in U$
and $n\geq 1$ such that $f^n(z)=p$.
It is clear that $(z,p)$ is mean proximal.
Then $\overline{PROX}(p)$ is a dense $G_\delta$ subset of $X$.
As the collection of transitive point is also a dense $G_\delta$ subset of $X$,
there exists a transitive point $x$ such that $(x,p)$ is mean proximal.
Now the result follows from Theorem~\ref{thm:mean-proximal-chaos}.
\end{proof}

We say that a dynamical system $(X,f)$ has the \emph{specification property}
if for any  $\varepsilon>0 $,  there exists a positive integer $N_{\varepsilon} >0$
such that for any integer $s\geqslant 2$,
any $s$ points  $y_1, y_2, \cdots, y_s  \in X$,
and any $2s$ integers $0=j_1\leqslant k_1<j_2\leqslant k_2<\cdots <j_s\leqslant k_s $
with $j_{l+1}-k_l\geqslant N_{\varepsilon}$ for $l=1, 2, \cdots, s-1 $,
there is a point $x$ in $X$ such that,  for each interval
$[j_m, k_m]$,  $d(f^i(x), f^i(y_m) )<\varepsilon,  i\in
[j_m, k_m], m=1, 2, \cdots, s$.
It is easy to see that if $(X,f)$ has the specification property and
$f$ is surjective,
then $(X,f)$ is strongly mixing.

\begin{prop}\label{prop:specification-unform-mean-chaos}
Let $(X,f)$ be a non-trivial dynamical system with $f$ being surjective.
If $(X,f)$ has a fixed point $p$ and satisfies the specification property, 
then there exists a dense, $f$-invariant,
uniformly mean chaotic set in $X$.
\end{prop}
\begin{proof}
As $(X,f)$ has the specification property,
it is not hard to see that for every $x\in X$
the mean proximal cell $\overline{PROX}(x)$ is a dense $G_\delta$ subset of $X$
(see e.g. \cite[Lemma7.4]{T16}).
In particular, for the fixed point $p$,
the mean proximal cell $\overline{PROX}(p)$ is a dense $G_\delta$ subset of $X$.
So there exists a transitive point $x$ such that $(x,p)$ is mean proximal
and then the result follows from Theorem~\ref{thm:mean-proximal-chaos}.
\end{proof}

\subsection{Invariant \texorpdfstring{$\delta$}{delta}-scrambled sets}
Let $(X, f) $ be a dynamical system, $n\geqslant 2$ and $\delta>0$.
We say that an tuple $(x_1, x_2, \dotsc, x_n)\in X^n$ is $n$-$\delta$-scrambled
if it satisfies
\[ \liminf_{k\to \infty} \max_{1\leqslant i<j\leqslant n} d(f^k(x_i), f^k(x_j))=0 \]
\[ \limsup_{k\to \infty} \min_{1\leqslant i<j\leqslant n} d(f^k(x_i), f^k(x_j))\geq\delta. \]
A subset $S\subseteq X$ (with at least $n$-points)
 is called \emph{$n$-$\delta$-scrambled}
if every $n$ pairwise distinct points form an $n$-$\delta$-scrambled tuple.
Note that $2$-$\delta$-scrambled sets are
just the classical $\delta$-scrambled sets in the sense of Li and Yorke.

First we have the following result and note that
the case $n=2$ was proved in \cite[Theorem 4.3]{FHLO16}.

\begin{prop}\label{prop:uncontable-to-Cantor}
Let $(X,f)$ be a dynamical system, $n\geq 2$ and $\delta>0$.
If there exists an uncountable, $f$-invariant, $n$-$\delta$-scrambled set in $X$,
then there exists a subsystem $(Y,f)$ of $(X,f)$ such that
there exists a dense, $\sigma$-Cantor, $f$-invariant,
$n$-$\delta$-scrambled set in $Y$.
\end{prop}
\begin{proof}
Recall that
\begin{align*}
PROX_n(f)&=\Bigl\{(x_1,x_2,\dotsc,x_n)\in X^n\colon
\liminf_{k\to \infty} \max_{1\leqslant i<j\leqslant n} d(f^k(x_i), f^k(x_j))=0\Bigr\}\\
  &= \bigcap_{M=1}^\infty\bigcup_{k=1}^\infty
  \biggl\{(x_1,x_2,\dotsc,x_n)\in X^n\colon \max_{1\leqslant i<j\leqslant n} d(f^k(x_i), f^k(x_j))<\frac{1}{M}\biggr\},
\end{align*}
and let
\begin{align*}
SEP_n(f,\delta)&=\Bigl\{(x_1,x_2,\dotsc,x_n)\in X^n\colon
\limsup_{k\to \infty} \min_{1\leqslant i<j\leqslant n}
d(f^k(x_i), f^k(x_j))\geq\delta\Bigr\}\\
&= \bigcap_{M=1}^\infty\bigcap_{N=1}^\infty\bigcup_{k=N}^\infty
  \biggl\{(x_1,x_2,\dotsc,x_n)\in X^n\colon
  \min_{1\leqslant i<j\leqslant n}
d(f^k(x_i), f^k(x_j))>\delta-\frac{1}{M}\biggr\},
\end{align*}
It is easy to see that both $PROX_n(f)$ and $SEP_n(f,\delta)$
are $G_\delta$ subsets of $X^n$.
Let $R=PROX_n(f)\cap SEP_n(f,\delta)$.
Then a subset $A$ of $X$ is $n$-$\delta$-scrambled if and only if it is $R$-dependent.
If there exists an uncountable, $f$-invariant, $n$-$\delta$-scrambled set in $X$,
applying Corollary~\ref{cor:local-f-K-W-Thm} to $\mathbf{C}(R)$,
there exists a Cantor set $A$ such that
$\bigcup_{i=0}^\infty f^i(A)$ is $n$-$\delta$-scrambled.
Let $Y=\overline{\bigcup_{i=0}^\infty f^i(A)}$.
Then the subsystem $(Y,f)$ is as required.
\end{proof}

We have the following characterization of existence of
invariant $n$-$\delta$-scrambled sets in any transitive system with a fixed point.
Note that the case $n=2$ was essentially proved in \cite[Theorem B]{FHLO16}.
\begin{thm}\label{thm:n-delta-scrambed}
Assume that $(X,f)$ be a non-trivial dynamical system and $n\geq 2$.
If $(X,f)$ is  transitive and has a fixed point,
then  the following conditions are equivalent:
\begin{enumerate}
  \item there exists a dense, $f$-invariant,
$n$-$\delta$-scrambled, uniformly chaotic set in $X$ for some $\delta>0$;

\item there exists $\delta'>0$ such that
for any integers $0=k_1<k_2<k_3<\dotsb<k_{n}$
there exists $z\in X$ satisfying $d(f^{k_i}(z),f^{k_j}(z))\geq \delta'$
for all $1\leq i<j\leq n$.
\end{enumerate}
\end{thm}

\begin{proof}
(1)$\Rightarrow$(2)
Let $S$ be an $f$-invariant $n$-$\delta$-scrambled set.
Note that $S$ contains at most one periodic point.
Hence we can pick up a point $x\in S$ which is not a periodic point.
Let $\delta'=\frac{\delta}{2}$.
For any $0=k_1<k_2<k_3<\dotsb<k_{n}$,
the tuple $(f^{k_1}(x),f^{k_2}(x),\dotsc,f^{k_n}(x))$ is  $n$-$\delta$-scrambled.
Then there exists $m\geq 0$ such that
$d(f^m(f^{k_i}(x)), f^m(f^{k_j}(x)))\geq \frac{\delta}{2}=\delta'$
for all $1\leq i<j\leq n$.
Let $z=f^m(x)$. Then we have $d(f^{k_i}(z),f^{k_j}(z))\geq \delta'$
for all $1\leq i<j\leq n$.

(2)$\Rightarrow$(1)
Let $\delta=\delta'$ and
\[SEP_n(f,\delta)=\Bigl\{(x_1,x_2,\dotsc,x_n)\in X^n\colon
\limsup_{k\to \infty} \min_{1\leqslant i<j\leqslant n}
d(f^k(x_i), f^k(x_j))\geq\delta\Bigr\}.\]
Let $Q=Q(RECUR,f)\cap Q(PROX,f)\cap \mathbf{C}(SEP_n(f,\delta))$ and
$\alpha_Q=\{R_n^Q\}_{n\in\mathbb{N}}$ be the coherent list on $X$
associated with $Q$.
Note that $SEP_n(f,\delta)$ is a $G_\delta$ subset of $X^n$ and
then $Q$ is a $G_\delta$, hereditary subset of $C(X)$.
Pick a transitive point $x\in X$ and put $A=Orb(x,f)$.
We claim that $A$ is in $SEP_n(f,\delta)$.
It is sufficient to show that $(f^{p_1}(x),f^{p_2}(x),\dotsc,f^{p_n}(x))\in R_n$
for all $0\leq p_1<p_2<p_3<\dotsb<p_{n}$.
Let $k_1=0$ and $k_i=p_{i}-p_{i-1}$ for $i=2,\dotsc,n$.
There exists $z\in X$ satisfying $d(f^{k_i}(z),f^{k_j}(z))\geq \delta'$
for all $1\leq i<j\leq n$.
As $x$ is a transitive point,
there exists an increasing sequence  $\{m_t\}$ of positive integers
such that $\lim_{t\to\infty}f^{m_t}(f^{p_1}(x))=z$.
By the continuity of $f$, $\lim_{t\to\infty}f^{m_t}(f^{p_i}(x))=f^{p_i-p_1}(z)=f^{k_i}(z)$
for $i=1,2,\dotsc,n$.
Then for $1\leq i<j\leq n$,
\[\lim_{t\to\infty}d(f^{m_t}(f^{p_i}(x)),f^{m_t}(f^{p_j}(x)))=d(f^{k_i}(z),f^{k_j}(z))
\geq \delta'=\delta,\]
which implies that
$(f^{p_1}(x),f^{p_2}(x),\dotsc,f^{p_n}(x))\in  SEP_n(f,\delta)$.
As it is proved in~Theorem~\ref{thm:uniform-chaos},
every finite subset of $A$ is in $Q(RECUR,f)\cap Q(PROX,f)$.
Then $A$ is $\alpha_Q$-dependent and
the condition (4) of Theorem~\ref{thm:Main-resut-KWT} is satisfied.
By the condition (2) of  Theorem~\ref{thm:Main-resut-KWT},
there exists a dense sequence $\{A_i\}$ in $CANTOR(X)$
such that $\bigcup_{j=0}^N\bigcup_{i=1}^N f^j(A_i)\in Q$ for all $N\geq 1$.
Let $C_k=\bigcup_{j=0}^k\bigcup_{i=1}^k f^j(A_i)$ and $C=\bigcup_{k=1}^\infty C_k$.
Then $C$ is a dense,  $f$-invariant,
$n$-$\delta'$-scrambled, uniformly chaotic set.
\end{proof}

Recall that a dynamical system $(X,f)$ is \emph{uniformly rigid}
if $X$ is a uniformly recurrent set.
It is easy to see that a dynamical system $(X,f)$ is not uniformly rigid
if and only if there exists $\delta>0$ such that for any $n\geq 1$ there exists
$z_n\in X$ with $d(f^n(z_n),z_n)\geq\delta$.
It is proved in \cite[Theorem C]{FHLO16} that
if a non-trivial compact dynamical system $(X, f)$ is transitive then $(X, f)$
contains a dense, $\sigma$-Cantor, $f$-invariant $\delta$-scrambled set
for some $\delta$ if and only
if it has a fixed point and is not uniformly rigid.
By Theorem~\ref{thm:n-delta-scrambed}, we can strengthen this result as follows.

\begin{cor}
If a non-trivial dynamical system $(X,f)$ is transitive and has a fixed point,
then there exists
a dense, $f$-invariant,
$\delta$-scrambled, uniformly chaotic set in $X$ for some $\delta>0$
if and only it $(X,f)$ is not uniformly rigid.
\end{cor}
As every strongly mixing system can never be uniformly rigid,
the following result is immediate.
Note that this result was first proved in \cite[Theorem 4]{BGO10}
without the conclusion of uniformly chaotic set.
\begin{cor}
If a non-trivial dynamical system $(X,f)$ is strongly mixing and has a fixed point,
then there exists a dense, $f$-invariant,
 $\delta$-scrambled, uniformly chaotic set in $X$ for some $\delta>0$.
\end{cor}

Let $\Lambda_m = \{0, 1,\dotsc,m-1\}$ be equipped with the discrete topology.
Let $\Sigma_m =\Lambda_m^{\mathbb{Z}_+}$ denote the set of
all infinite sequence of symbols in $\Lambda_m$
indexed by the non-negative integers $\mathbb{Z}_+$ with the product topology.
A compatible metric on  $\Sigma_m$ is defined as follows: for any two points $x,y\in\Sigma_m $,
\[ d(x,y)=\left\lbrace \begin{array}{ll}
0,& \text{ if } x=y,\\ \dfrac{1}{k+1},& \text{ if } x\neq y \text{ and } k=\min\{i\geq 0,x_i\neq y_i \}.
\end{array}\right.  \]
The \emph{shift} transformation is a continuous map
$\sigma:\Sigma_m\to\Sigma_m$ given by
$(\sigma(x))_i=x_{i+1}$, where $x=x_0x_1x_2\dotsb$.
It is clear that $\sigma$ is a continuous surjection.
The dynamical system $(\Sigma_m,\sigma)$ is called the \emph{full shift}
on $m$ symbols.

\begin{prop}
For the full shift $(\Sigma_2,\sigma)$,
 there exists a dense, $f$-invariant,
 $n$-$\frac{1}{n-1}$-scrambled, uniformly chaotic set in $X$ for all $n\geq 2$.
\end{prop}
\begin{proof}
According to Theorem~\ref{thm:n-delta-scrambed},
it is sufficient to prove that for any $ n\geq2 $ and any integers $ 0=k_1<k_2<\cdots<k_n $
there exists $z\in \Sigma_2$ such that the words
$z[k_i,k_i+n-2]$, $i=1,2,\dotsc,n$, are pairwise distinct.
We prove this by induction on $n$.

For $n=2$ and $0=k_1<k_2$,
we define a point $z\in \Sigma_2$ in the following way:
$z[0]=0$, $z[k_2]=1$ and other undetermined coordinates can be defined freely.

Assume the result has been established for $n=m-1 $ and consider the case $n=m$.
Fix $m$ integers $0=k_1<k_2<\cdots<k_{m-1}<k_m$.
For $0=k_1<k_2<\cdots<k_{m-1}$, by the assumption
there exists a point $x\in \Sigma_2$
such that $x[k_i,k_i+m-3]$, $i=1,2,\dotsc,m-1$, are pairwise distinct.
We define a point $z\in \Sigma_2$ as follows.
First for $i=1,2,\dotsc,m-1$, let $z[k_i,k_i+m-2]=x[k_i,k_i+m-2]$.
If $ k_m\geq k_{m-1}+m-1$,
then we define $ z[k_m,k_m+m-2]$ be any word which
is different from all the words $z[k_i,k_i+m-2]$, $i=1,2,\dotsc,m-1$.
Otherwise, there exists $ 1\leq l\leq m-2$ such that $k_m=k_{m-1}+l$.
In this case, $z[k_m,k_{m-1}+m-2]$ has been already defined.
If there exists $1\leq i_0\leq m-1$ such that
$z[k_{i_0}, k_{i_0}+m-l-2]=z[k_m,k_{m-1}+m-2]$,
then we define $z[k_m,k_m+m-3]=z[k_{i_0},k_{i_0}+m-3]$
and $z[k_m+m-2]$ is different from $z[k_{i_0}+m-2]$.
If there is no such $i_0$ then we can define $z[k_{m-1}+m,k_m+m-2]$ freely.
All in all, the words $z[k_i,k_i+m-2]$, $i=1,2,\dotsc,m$, are pairwise distinct.
\end{proof}

\begin{prop}\label{prop:exact-delta-chaos}
If a non-trivial compact dynamical system $(X,f)$ is exact
and has a fixed point, then there exists a dense, $f$-invariant,
$3$-$\delta$-scrambled, uniformly mean chaotic set in $X$ for some $\delta>0$.
\end{prop}
\begin{proof}
As $(X,f)$ is exact, $(X,f)$ is strongly mixing and then $X$ is perfect as $X$ is not a singleton.

According to Theorem~\ref{thm:n-delta-scrambed}, it is enough to prove that
there exists $\delta>0$ such that  for any $0=k_1<k_2<k_3$ there exists $z\in X$
such that $d(f^{k_i}(z),f^{k_j}(z))\geq \delta$ 
for all $1\leq i<j\leq 3$.

Let $p\in X$ be a fixed point.
Pick  $3$ pairwise distinct points $x_1,x_2,x_3$ in $X$ with $x_3=p$.
Let $\delta_1=\frac{1}{3}\min\{d(x_i,x_j)\colon 1\leq i<j\leq 3\}$
and $U_i=B(x_i,\delta_1)$ for $i=1,2,3$.
As $(X,f)$ is exact, there exists $N\geq 1$ such that $f^N(U_i)=X$ for $i=1,2,3$.

Fix $3$ integers $0=k_1<k_2<k_3$. We have the following 4 cases.

Case 1: $k_2-k_1\geq N$ and $k_3-k_2\geq N$.
As $f^n(U_i)=X$ for any $n\geq N$ and $i=1,2$,
there exists a point $z\in U_1$ such that $f^{k_i}(z)\in U_i$ for $i=2,3$.
Then $d(f^{k_i}(z),f^{k_j}(z))\geq \delta_1$
for all $1\leq i<j\leq 3$.

Case 2: $k_2-k_1< N$ and $1\leq k_3-k_2<N$.
As $f^i(p)=p\in U_3$ for $i=0,1,\dotsc,2N$,
there exists a non-empty open subset $U_3'$ of $U_3$
such that $f^i(U_3')\subset U_3$ for $i=0,1,\dotsc,2N$.
As $(X,f)$ is strongly mixing and $X$ is perfect, pick a transitive point $y\in U_3'$.
Then $f^i(y)\neq f^j(y)$ for all $0\leq i<j$.
Let $\delta_2=\min\{d(f^i(y),f^j(y))\colon 0\leq i<j\leq 2N\}>0$.
Then  $d(f^{k_i}(y),f^{k_j}(y))\geq \delta_2$
for all $1\leq i<j\leq 3$.

Case 3: $k_2-k_1\geq N$ and $1\leq k_3-k_2<N$.
Let $y$ as in Case 2.
As $f^n(U_1)=X$ for all $n\geq N$,
there exists a point $z\in U_1$ such that $f^{k_2}(z)=y$.
Note that $f^{k_3}(z)=f^{k_3-k_2}(y)\in U_3$.
So $d(f^{k_i}(z),f^{k_j}(z))\geq \min\{\delta_1,\delta_2\}$
for all $1\leq i<j\leq 3$.

Case 4: $k_2-k_1< N$ and $k_3-k_2\geq N$.
Similar to the Case 1, there exists a point $z\in U_1$ such that
$f^{N\cdot k_2}(z)\in U_2$ and $f^{k_3+(N-1)k_2}(z)\in U_3'$.
As $X$ is compact, $f$ is uniformly continuous,
there exists $\delta_3>0$ such that if $d(u,v)<\delta_3$
then $d(f^i(u),f^i(v))<\delta_1/N$ for $i=0,1,2,\dotsc,N^2$.
We claim that $d(f^{k_i}(z),f^{k_j}(z))\geq \delta_3$ for all $1\leq i<j\leq 3$.
We prove this claim by contradiction.
\begin{itemize}
\item If $d(f^{k_2}(z),f^{k_3}(z))<\delta_3$,
then $d(f^{k_2+(N-1)k_2}(z),f^{k_3+(N-1)k_2}(z))<\delta_1/N<\delta_1$.
As $f^{N\cdot k_2}(z)\in U_2$ and $f^{k_3+(N-1)k_2}(z)\in U_3'\subset U_3$,
we have $d(f^{N\cdot k_2}(z),f^{k_3+(N-1)k_2}(z))>\delta_1$.
This is a contradiction. Thus
 $d(f^{k_2}(z),f^{k_3}(z))\geq \delta_3$.
\item If $d(z,f^{k_3}(z))<\delta_3$,
then $d(f^{N\cdot k_2}(z),f^{k_3+N\cdot k_2}(z))<\delta_1/N<\delta_1$.
Note that $f^{N\cdot k_2}(z)\in U_2$ and
$f^{k_3+N\cdot k_2}(z))=f^{k_2}(f^{k_3+(N-1)k_2}(z))\in f^{k_2}(U_3')\subset U_3$,
we have $d(f^{N\cdot k_2}(z),f^{k_3+N\cdot k_2}(z))>\delta_1$.
This is a contradiction. Thus
$d(z,f^{k_3}(z))\geq \delta_3$.
\item If $d(z,f^{k_2}(z))<\delta_3$, then
\begin{align*}
  d(f^{k_2}(z),f^{2k_2}(z))<\delta_1/N,\\
  d(f^{2k_2}(z),f^{3k_2}(z))<\delta_1/N,\\
  \dotsb\\
  d(f^{(N-1)k_2}(z),f^{N\cdot k_2}(z))<\delta_1/N.\\
\end{align*}
Note that $\delta_3\leq \delta_1/N$ and
\[d(z,f^{N\cdot k_2}(z))\leq
\sum_{i=0}^{N-1} d(f^{i\cdot k_2}(z),f^{(i+1)\cdot k_2}(z))< \delta_1.\]
But $z\in U_1$, $f^{N\cdot k_2}(z)\in U_2$, $d(z,f^{N\cdot k_2}(z))>\delta_1$.
This is also a contradiction. Then $d(z,f^{k_2}(z))\geq \delta_3$.
\end{itemize}
All in all, let $\delta=\min\{\delta_1,\delta_2,\delta_3\}$.
Then for any $0=k_1<k_2<k_3$ there exists $z\in X$
such that $d(f^{k_i}(z),f^{k_j}(z))\geq \delta_1$
for all $1\leq i<j\leq 3$.
\end{proof}

We strongly believe that there is a multi-variant version of Proposition~\ref{prop:exact-delta-chaos},
but now we can not prove this.
We state the following question for the further study.
\begin{que}
If a non-trivial compact dynamical system $(X,f)$ is exact or strongly mixing and has a fixed point,
does there exist a dense, $f$-invariant,
$n$-$\delta_n$-scrambled, uniformly chaotic set in $X$ for all $n\geq 2$
and some $\delta_n>0$.
\end{que}

\subsection{Invariant distributional \texorpdfstring{$\delta$}{delta}-scrambled sets}
Let $(X,f)$ be a dynamical system.
For $x_1,x_2,\ldots,x_n\in X$, $t>0$ and $n\geq 2$, put
\[\Phi_{(x_1,x_2,\ldots,x_n)}(t)=
\liminf_{m\to\infty}\frac{1}{m}
\#\Bigl\{1\le k\le m: \min_{1\leq i<j\leq n}d\bigl(f^{k}(x_i), f^{k}(x_j)\bigr) \leq t\Bigr\}\]
and
\[\Phi^*_{(x_1,x_2,\ldots,x_n)}(t)=\limsup_{m\to\infty}\frac{1}{m}
\#\Bigl\{1\le k\le m: \max_{1\leq i<j\leq n}d\bigl(f^{k}(x_i), f^{k}(x_j)\bigr) < t\Bigr\}.\]

For a given $\delta>0$, a tuple $(x_1,x_2,\ldots,x_n)\in X^n$
is called \emph{distributionally $n$-$\delta$-scrambled} if
 $\Phi^*_{(x_1,x_2,\ldots,x_n)}(t)=1$ for every $t>0$ and
 $\Phi_{(x_1,x_2,\ldots,x_n)}(\delta)=0$.
A subset $S$ (with at least $n$ points) of $X$
is called \emph{distributionally $n$-$\delta$-scrambled}
if every $n$ pairwise distinct points in $S$ form a
distributionally $n$-$\delta$-scrambled tuple.
Note that distributionally  $2$-$\delta$-scrambled sets are
just the classical distributionally  $\delta$-scrambled sets.

It should be noticed that by Lemma~\ref{lem:mean-proximal}
a tuple $(x_1,x_2,\ldots,x_n)\in X^n$ is mean proximal if and only if
$\Phi^*_{(x_1,x_2,\ldots,x_n)}(t)=1$ for every $t>0$.

\begin{prop}\label{prop:distributionally-delta-scrambled-set}
Let $(X,f)$ be a dynamical system, $n\geq 2$ and $\delta>0$.
If there exists an uncountable, $f$-invariant,
distributionally $n$-$\delta$-scrambled set in $X$,
then there exists a subsystem $(Y,f)$ of $(X,f)$ such that
there exists a dense, $\sigma$-Cantor, $f$-invariant,
distributionally  $n$-$\delta$-scrambled set in $Y$.
\end{prop}
\begin{proof}
By the proof of Theorem~\ref{thm:mean-proximal-chaos},
$Q(\overline{PROX},f)$ is a $G_\delta$ subset of $C(X)$. 
Then by Remark~\ref{rem:G-delta-coherent-list}, 
$\overline{PROX}_n(f)$ is a $G_\delta$ subset of $X^n$.
Let
\[\overline{SEP}_n(f,\delta)=\bigl\{(x_1,x_2,\dotsc,x_n)\in X^n\colon
\Phi_{(x_1,x_2,\ldots,x_n)}(\delta)=0\bigr\}.\]
Then
\begin{align*}
\overline{SEP}_n(f,\delta)
&=\biggl\{(x_1,x_2,\dotsc,x_n)\in X^n\colon\\
&\qquad \limsup_{m\to\infty}\frac{1}{m}
\#\Bigl\{1\le k\le m: \min_{1\leq i<j\leq n}d\bigl(f^{k}(x_i), f^{k}(x_j)\bigr) > \delta\Bigr\}=1
\biggr\}\\
&=\bigcap_{M=1}^\infty \bigcap_{N=1}^\infty\bigcup_{m=N}^\infty
\biggl\{(x_1,x_2,\dotsc,x_n)\in X^n\colon\\
&\qquad
\#\Bigl\{1\le k\le m: \min_{1\leq i<j\leq n}d\bigl(f^{k}(x_i), f^{k}(x_j)\bigr) >\delta\Bigr\}
\geq \Bigl(1-\frac{1}{M}\Bigr)m\biggr\}.
\end{align*}
Now it is easy to see that
$\overline{SEP}_n(f,\delta)$ is a $G_\delta$ subset of $X^n$.
Let $R=\overline{PROX}_n(f)\cap \overline{SEP}_n(f,\delta)$.
Then a subset $A$ of $X$ is distributionally
$n$-$\delta$-scrambled if and only if it is $R$-dependent.
Now the proof is similar to that of Proposition~\ref{prop:uncontable-to-Cantor}.
\end{proof}

We have the following criterion for the existence of
invariant distributionally $n$-$\delta$-scrambled sets.
\begin{thm}\label{thm:inv-distri-delta-chaos}
Assume that $(X,f)$ is a non-trivial compact dynamical system and $n\geq 2$.
If $(X,f)$ is transitive and has a fixed point $p$, and
 for every minimal point $z\in X$ the mean proximal cell $\overline{PROX}(z)$
is dense in $X$,
then the following conditions are equivalent:
\begin{enumerate}
\item there exists a dense, $f$-invariant,
distributionally $n$-$\delta$-scrambled, uniformly mean chaotic set in $X$;
\item there exists $\delta'>0$ such that for any  integers
$0=k_1<k_2<k_3<\dotsb<k_{n}$ there exists $z\in X$ satisfying
$\inf_{m\geq 1} d(f^{k_i+m}(z),f^{k_j+m}(z))\geq \delta'$
for all $1\leq i<j\leq n$.
\end{enumerate}
\end{thm}
\begin{proof}
(1)$\Rightarrow$(2)
 Let $S$ be a $f$-invariant distributionally $n$-$\delta$-scrambled set.
Note that $S$ contains at most one periodic point.
Hence we can pick up a point $x\in S$ which is not a periodic point.
Let $\delta'=\frac{\delta}{2}$.
For any $0=k_1<k_2<k_3<\dotsb<k_{n}$,
the tuple $(f^{k_1}(x),f^{k_2}(x),\dotsc,f^{k_n}(x))$ is
distributionally $n$-$\delta$-scrambled.
By the proof of Proposition~\ref{prop:distributionally-delta-scrambled-set},
we have that the upper density of the set
$\{m\in\mathbb{N}\colon \allowbreak
d(f^m(f^{k_i}(x)), f^m(f^{k_j}(x)))\geq \frac{\delta}{2}=\delta'$
for all $1\leq i<j\leq n\}$ is one.
As every subset of $\mathbb{N}$
with upper density one has arbitrary finite long intervals in $\mathbb{N}$,
there exists an increasing sequence $\{q_s\}_{s=1}^\infty$ of positive integers
such that $d(f^m(f^{k_i}(x)), f^m(f^{k_j}(x)))\geq \frac{\delta}{2}=\delta'$
for all $1\leq i<j\leq n$, $m\in [q_s,q_s+s]$ and $s\geq 1$.
By the compactness of $X$, without loss of generality,
assume that $f^{q_s}(x)\to z$ as $s\to\infty$.
For any $m\geq 1$ and $1\leq i<j\leq n$,
$d(f^{k_i+m}(z),f^{k_j+m}(z))=\lim_{s\to\infty}
d(f^{q_s+m}(f^{k_i}(x)), f^{q_s+m}(f^{k_j}(x)))\geq \delta'$.
Then $z$ is as required.

(2)$\Rightarrow$(1)
Note that there exists $\delta'>0$ such that for any  integers
$0=k_1<k_2<k_3<\dotsb<k_{n}$ there exists $z=z(k_1,k_2,\dotsc,k_n)\in X$ satisfying
$\inf_{m\geq 1} d(f^{k_i+m}(z),f^{k_j+m}(z))\geq \delta'$
for all $1\leq i<j\leq n$.
As $(X,f)$ is a compact dynamical system,
we can require that $z$ is a minimal point by replacing $z$ with
any minimal point in the orbit closure of $z$.

Let $\delta=\frac{\delta'}{2}$ and
\[\overline{SEP}_n(f,\delta)=\bigl\{(x_1,x_2,\dotsc,x_n)\in X^n\colon
 \Phi_{(x_1,x_2,\ldots,x_n)}(\delta)=0\bigr\}.\]
Let
\[Q=Q(RECUR,f)\cap Q(\overline{PROX},f)\cap \mathbf{C}(\overline{SEP}_n(f,\delta))\]
and $\alpha_Q=\{R_n^Q\}_{n\in\mathbb{N}}$ be the coherent list on $X$
associated with $Q$.
As $\overline{SEP}_n(f,\delta)$ is a $G_\delta$ subset of $X^n$,
$Q$ is a hereditary $G_\delta$ subset of $C(X)$.

As every mean proximal cell is $G_\delta$,
for each minimal point $z\in X$
the mean proximal cell $\overline{PROX}(z)$ is a dense $G_\delta$ subset of $X$.
Pick a transitive point
\[x\in \overline{PROX}(p)\cap \bigcap_{0=k_1<k_2<k_3<\dotsb<k_{n}}
\bigcap_{i\geq 0}\overline{PROX}\bigl(f^i(z(k_1,k_2,\dotsc,k_n))\bigr)\]
 and put $A=Orb(x,f)$.
We claim that $A$ is $\alpha_Q$-dependent.
As it is proved in~Theorem~\ref{thm:mean-proximal-chaos},
every finite subset of $A$ is in $Q(RECUR,f)\cap Q(\overline{PROX},f)$.
It is sufficient to show that
$(f^{p_1}(x),f^{p_2}(x),\dotsc,f^{p_n}(x))\in SEP_n(f,\delta)$
for all $0\leq p_1<p_2<p_3<\dotsb<p_{n}$.
Let $k_i=p_{i}-p_{1}$ for $i=1,2,\dotsc,n$.
There exists $z=z(k_1,k_2,\dotsc,k_n)\in X$ satisfying
$\inf_{m\geq 1} d(f^{k_i+m}(z),f^{k_j+m}(z))\geq \delta'$
for all $1\leq i<j\leq n$.
As $(x,f^{p_1}(z))$ is mean proximal,
by Lemma~\ref{lem:mean-proximal}
there exists an increasing sequence  $\{m_t\}$ of positive integers
with upper dense one
such that $\lim_{t\to\infty}d(f^{m_t}(x),f^{m_t}(f^{p_1}(z)))=0$.
By the continuity of $f$,
$\lim_{t\to\infty}d(f^{m_t}(f^{p_i}(x)),f^{m_t}(f^{k_i}(z)))=0$ for $i=1,2,\dotsc,n$,
as $f^{p_i-p_1}(z)=f^{k_i}(z)$.
By the proof of Proposition~\ref{prop:distributionally-delta-scrambled-set}, we have
$\Phi_{(f^{p_1}(x),f^{p_2}(x),\dotsc,f^{p_n}(x))}(\delta)=0$.
Then $A$ is $\alpha_Q$-dependent and the condition (4) of Theorem~\ref{thm:Main-resut-KWT} is satisfied.
By the condition (2) of  Theorem~\ref{thm:Main-resut-KWT},
there exists a dense sequence $\{A_i\}$ in $CANTOR(X)$
such that $\bigcup_{j=0}^N\bigcup_{i=1}^N f^j(A_i)\in Q$ for all $N\geq 1$.
Let $C_k=\bigcup_{j=0}^k\bigcup_{i=1}^k f^j(A_i)$ and $C=\bigcup_{k=1}^\infty C_k$.
Then $C$ is a dense, $f$-invariant,
distributionally $n$-$\delta$-scrambled,
uniformly mean chaotic set in $X$.
\end{proof}

\begin{prop}
For the full shift $(\Sigma_n,\sigma)$ on $n$ symbols,
there exists a dense, $f$-invariant,
distributionally $n$-$1$-scrambled, uniformly mean chaotic set in $X$.
\end{prop}
\begin{proof}
According to Theorem~\ref{thm:inv-distri-delta-chaos},
 it is sufficient to prove that for any integers $ 0=k_1<k_2<\cdots<k_n $
there exists $z\in \Sigma_n$ such that for any $j\geq 1$,
the words $z[j+k_i]$, $i=1,2,\dotsc,n$, are pairwise distinct.

Fix $0\leq k_1<k_2<\dotsb<k_n$.
As each position in a point has $n$ choices  of the symbols,
we define a point $z\in \Sigma_2$ in the following way:
\begin{itemize}
  \item $z[i]=0$ for any $0\leq i<k_2$,
  \item $z[i]$ is different from $z[i-k_2]$ for $k_2\leq i<k_3$,
  \item ...
  \item $z[i]$ is different from all $z[i-k_2]$, $z[i-k_3]$, $\dotsc$,
  $z[i-k_n]$ for  $i\geq k_n$.
\end{itemize}
By the construction, it is easy to check that $z$ is as required.
\end{proof}

It is shown in~\cite[Theorem 16]{FOW14} that
if a non-trivial compact dynamical system $(X,f)$ satisfies the specification property
and has a fixed point, then there exists a dense, $f$-invariant,
distributionally $\delta$-scrambled set in $X$ for some $\delta>0$.
By Theorem~\ref{thm:inv-distri-delta-chaos}, we can strengthen this result as follows.
\begin{prop}
Let $(X,f)$ be a non-trivial compact dynamical system with $f$ being surjective.
If $(X,f)$ satisfies the specification property and has a fixed point $p$,
 then there exists a dense, $f$-invariant,
distributionally $\delta$-scrambled,
uniformly mean chaotic set in $X$ for some $\delta>0$.
\end{prop}
\begin{proof}
By \cite[Lemma~14]{FOW14}, there exists $\delta'>0$
such that for any $n\geq 1$ there exists $z_n\in X$
satisfying $\inf_{i\geq 0} d(f^i(z_n),f^i(f^n(z_n))\geq \delta'$.
As said in Proposition~\ref{prop:specification-unform-mean-chaos},
for a dynamical system with the specification property,  every mean proximal cell is dense.
So the result follows from Theorem~\ref{thm:inv-distri-delta-chaos}.
\end{proof}

We conjecture that there is a multi-variant version of the above result and
end this paper with the following question.
\begin{que}
If a non-trivial dynamical system $(X,f)$ is has
the specification property and has a fixed point,
does there exist a  dense, $f$-invariant,
distributionally $n$-$\delta_n$-scrambled, uniformly mean chaotic set in $X$
 for all $n\geq 2$
and some $\delta_n>0$.
\end{que}

\section*{Acknowledgments}
The authors were supported in part by NSF of China (grant numbers 11771264 and 11471125).
The authors would like to thank Feng Tan for numerous discussions
on the topics covered by the paper.

\bibliographystyle{amsplain}

\begin{thebibliography}{10}

\bibitem{A04}
Ethan Akin, \emph{Lectures on {C}antor and {M}ycielski sets for dynamical
  systems}, Chapel {H}ill {E}rgodic {T}heory {W}orkshops, Contemp. Math., vol.
  356, Amer. Math. Soc., Providence, RI, 2004, pp.~21--79. \MR{2087588}

\bibitem{AGHSY10}
Ethan Akin, Eli Glasner, Wen Huang, Song Shao, and Xiangdong Ye,
  \emph{Sufficient conditions under which a transitive system is chaotic},
  Ergodic Theory Dynam. Systems \textbf{30} (2010), no.~5, 1277--1310.
  \MR{2718894}

\bibitem{BGO10}
Francisco Balibrea, Juan L.~G. Guirao, and Piotr Oprocha, \emph{On invariant
  {$\varepsilon$}-scrambled sets}, Internat. J. Bifur. Chaos Appl. Sci. Engrg.
  \textbf{20} (2010), no.~9, 2925--2935. \MR{2738744}

\bibitem{BZ15}
Taras Banakh and Lyubomyr Zdomskyy, \emph{Non-meager free sets for meager
  relations on {P}olish spaces}, Proc. Amer. Math. Soc. \textbf{143} (2015),
  no.~6, 2719--2724. \MR{3326049}

\bibitem{BGKM02}
Fran\c{c}ois Blanchard, Eli Glasner, Sergi{\u\i} Kolyada, and Alejandro Maass,
  \emph{On {L}i-{Y}orke pairs}, J. Reine Angew. Math. \textbf{547} (2002),
  51--68. \MR{1900136}

\bibitem{BHS08}
Fran\c{c}ois Blanchard, Wen Huang, and L'ubom{\'\i}r Snoha, \emph{Topological
  size of scrambled sets}, Colloq. Math. \textbf{110} (2008), no.~2, 293--361.
  \MR{2353910}

\bibitem{BH87}
Andrew~M. Bruckner and Thakyin Hu, \emph{On scrambled sets for chaotic
  functions}, Trans. Amer. Math. Soc. \textbf{301} (1987), no.~1, 289--297.
  \MR{879574}

\bibitem{DK16}
Martin Dole\v{z}al and Wies{\l}aw Kubi{\'s}, \emph{Perfect independent sets
  with respect to infinitely many relations}, Arch. Math. Logic \textbf{55}
  (2016), no.~7-8, 847--856. \MR{3555328}

\bibitem{D05}
Bau-Sen Du, \emph{On the invariance of {L}i-{Y}orke chaos of interval maps}, J.
  Difference Equ. Appl. \textbf{11} (2005), no.~9, 823--828. \MR{2159799}

\bibitem{FHLO16}
Magdalena Fory\'s, Wen Huang, Jian Li, and Piotr Oprocha, \emph{Invariant
  scrambled sets, uniform rigidity and weak mixing}, Israel J. Math.
  \textbf{211} (2016), no.~1, 447--472. \MR{3474970}

\bibitem{FOW14}
Magdalena Fory\'s, Piotr Oprocha, and Pawe{\l} Wilczy\'nski, \emph{Factor maps
  and invariant distributional chaos}, J. Differential Equations \textbf{256}
  (2014), no.~2, 475--502. \MR{3121703}

\bibitem{HY02}
Wen Huang and Xiangdong Ye, \emph{Devaney's chaos or 2-scattering implies
  {L}i-{Y}orke's chaos}, Topology Appl. \textbf{117} (2002), no.~3, 259--272.
  \MR{1874089}
  
 \bibitem{I91}
 Anzelm Iwanik, \emph{Independence and scrambled sets for chaotic mappings}, The
  mathematical heritage of {C}. {F}. {G}auss, World Sci. Publ., River Edge, NJ,
  1991, pp.~372--378. \MR{1146241}

\bibitem{K95}
Alexander~S. Kechris, \emph{Classical descriptive set theory}, Graduate Texts
  in Mathematics, vol. 156, Springer-Verlag, New York, 1995. \MR{1321597}

\bibitem{K73}
Kazimierz Kuratowski, \emph{Applications of the {B}aire-category method to the
  problem of independent sets}, Fund. Math. \textbf{81} (1973), no.~1, 65--72.
  \MR{0339092}

\bibitem{LTY15}
Jian Li, Siming Tu, and Xiangdong Ye, \emph{Mean equicontinuity and mean
  sensitivity}, Ergodic Theory Dynam. Systems \textbf{35} (2015), no.~8,
  2587--2612. \MR{3456608}

\bibitem{LY16}
Jian Li and Xiangdong Ye, \emph{Recent development of chaos theory in
  topological dynamics}, Acta Math. Sin. (Engl. Ser.) \textbf{32} (2016),
  no.~1, 83--114. \MR{3431162}

\bibitem{LY75}
Tien-Yien Li and James~A. Yorke, \emph{Period three implies chaos}, Amer. Math.
  Monthly \textbf{82} (1975), no.~10, 985--992. \MR{0385028}

\bibitem{MRZ17}
Andrea Medini, Du{\v s}an Repov{\v s}, and Lyubomyr Zdomskyy, \emph{Non-meager
  free sets and independent families}, 
Proc. Amer. Math. Soc. \textbf{145} (2017), 4061--4073.

\bibitem{M64}
Jan Mycielski, \emph{Independent sets in topological algebras}, Fund. Math.
  \textbf{55} (1964), 139--147. \MR{0173645}

\bibitem{O11}
Piotr Oprocha, \emph{Coherent lists and chaotic sets}, Discrete Contin. Dyn.
  Syst. \textbf{31} (2011), no.~3, 797--825. \MR{2825640}

\bibitem{T16}
Feng Tan, 
\emph{On an extension of {M}ycielski's theorem and invariant
  scrambled sets}, Ergodic Theory Dynam. Systems \textbf{36} (2016), no.~2,
  632--648. \MR{3503038}

\bibitem{TXL07}
Feng Tan, Jincheng Xiong, and Jie L\"u, 
\emph{Dependent sets of a countable
  family of thick relations on a metric space}, 
Southeast Asian Bull. Math.
  \textbf{31} (2007), no.~6, 1077--1089. \MR{2386985}

\bibitem{W82}
Peter Walters, \emph{An introduction to ergodic theory}, Graduate Texts in
  Mathematics, vol.~79, Springer-Verlag, New York-Berlin, 1982. \MR{648108}

\bibitem{YL09}
Dalian Yuan and Jie L\"u, \emph{Invariant scrambled sets in transitive
  systems}, Adv. Math. (China) \textbf{38} (2009), no.~3, 302--308.
  \MR{2561797}

\end{thebibliography}

\end{document}